\documentclass[11pt]{amsart}
\usepackage{amssymb}
\usepackage{amsmath}
\usepackage{mathrsfs}
\usepackage{amsthm}
\usepackage{enumerate}
\usepackage{color}
\newtheorem{theorem}{Theorem}[section]
\newtheorem{lemma}[theorem]{Lemma}
\newtheorem{corollary}[theorem]{Corollary}
\newtheorem{conjecture}[theorem]{Conjecture}
\theoremstyle{definition}

\newtheorem{remark}[theorem]{Remark}

\numberwithin{equation}{section}
\mathchardef\hyphen="2D
\allowdisplaybreaks

\definecolor{AfonsoBlue}{RGB}{30,65,123}

\begin{document}

       \author{Afonso S. Bandeira}
       \address{Department of Mathematics, ETH Z\"urich, Switzerland}
       \email{bandeira@math.ethz.ch}

       % second author

       \author{March T. Boedihardjo}%Supported by NSF DMS-1856221.

       % the address where the research was carried out
       \address{Department of Mathematics, University of California, Los Angeles, USA}

       % current address, usually not needed because it is the same as the
       % regular address

       \email{march@math.ucla.edu}

\title{The spectral norm of Gaussian matrices with correlated entries}
%\author{Afonso \& March}

\begin{abstract}
We give a non-asymptotic bound on the spectral norm of a $d\times d$ matrix $X$ with centered jointly Gaussian entries in terms of the covariance matrix of the entries. In some cases, this estimate is sharp and removes the $\sqrt{\log d}$ factor in the noncommutative Khintchine inequality. {\bf This paper is superseded by https://arxiv.org/abs/2108.06312}
\end{abstract}

\maketitle

\section{Introduction}
Let $X$ be a $d\times d$ centered random matrix with (correlated) jointly Gaussian entries. We aim to provide an estimate for the expected spectral norm of $\mathbb{E}\|X\|$ in terms of the $d^{2}\times d^{2}$ covariance matrix $\mathbb{E}(X\otimes X)$ of the Gaussian entries. This problem is settled by the noncommutative Khintchine inequality \cite{Buchholz,LustPiquard,Oliveira} up to a $\sqrt{\log d}$ factor, namely,
\begin{equation}\label{Introeq1}
\|\mathbb{E}(X^{*}X)\|^{\frac{1}{2}}+\|\mathbb{E}(XX^{*})\|^{\frac{1}{2}}\lesssim\mathbb{E}\|X\|\lesssim \sqrt{\log d} \left(\|\mathbb{E}(X^{*}X)\|^{\frac{1}{2}}+\|\mathbb{E}(XX^{*})\|^{\frac{1}{2}}\right),
\end{equation}
where $\lesssim$ denotes smaller or equal up to multiplicative dimension-free constant.

The $\sqrt{\log d}$ factor on the right hand side of (\ref{Introeq1}) is, in general, required: if $X$ is diagonal with i.i.d.~standard Gaussian diagonal entries, then $\mathbb{E}\|X\|\sim\sqrt{\log d}$ and
$\|\mathbb{E}(X^{*}X)\|=\|\mathbb{E}(XX^{*})\|=1$.
%$\|\mathbb{E}(X^{*}X)\|^{\frac{1}{2}}=\|\mathbb{E}(XX^{*})\|^{\frac{1}{2}}=1$.
By contrast, if the $d^{2}$ entries of $X$ are i.i.d.~standard Gaussian random variables, then $\mathbb{E}\|X\|\sim\sqrt{d}$ and $\|\mathbb{E}(X^{*}X)\|=\|\mathbb{E}(XX^{*})\|=d$
%$\|\mathbb{E}(X^{*}X)\|^{\frac{1}{2}}=\|\mathbb{E}(XX^{*})\|^{\frac{1}{2}}=\sqrt{d}$
so in this case, the $\sqrt{\log d}$ factor can be removed. More generally, if the entries of $X$ are independent and the variances of the entries are homogeneous enough, then the $\sqrt{\log d}$ factor can be removed \cite{BandeiravanHandel,Ramon-inhomogeneous1,Ramon-inhomogeneous2}.

Estimates for the spectral norm of random matrices are a central tool in both pure and applied mathematics, we point the interested reader to the monograph~\cite{TroppMatrixConcentration} and references therein for applications. We note also that the extra dimensional factor often propagates to the applications resulting in suboptimal bounds. %\asb{Short paragraph about importance of these bounds with reference to Tropp's monograph and issues with the dimensional factors}

The extent to which the $\sqrt{\log d}$ factor can be removed in (\ref{Introeq1}), in general, is mostly unknown. A notable result in this direction, whose insights we build on, is the work of Tropp~\cite{Tropp2ndorder} which introduces a quantity $w(X)$, for a self-adjoint Gaussian matrix $X$, and shows that
\[\mathbb{E}\|X\|\lesssim \sqrt[4]{\log d}\|\mathbb{E}(X^{2})\|^{\frac12}+\sqrt{\log d}\cdot w(X)\]
for all (correlated) self-adjoint Gaussian matrices $X$. %In other words, with a tradeoff of $\sqrt{\log d}\cdot w(X)$, one can reduce the $\sqrt{\log d}$ factor in (\ref{Introeq1}) to $(\log d)^{\frac{1}{4}}$.
When all the $d^{2}$ entries of $X$ are i.i.d.~standard Gaussian, this estimate improves (\ref{Introeq1}) but is still not sharp because of the $\sqrt[4]{\log d}$ factor. Moreover, in general, computing $w(X)$ directly appears to be challenging. %Our main theorem, which we present now, builds on insights from~\cite{Tropp2ndorder}.

The following is the main result of this paper. % where we give an estimate of $\mathbb{E}\|X\|$ in terms of the spectral norms $\|\mathbb{E}(X^{*}X)\|,\|\mathbb{E}(XX^{*})\|,\|\mathbb{E}(X\otimes X)\|$ for all $d\times d$ centered correlated Gaussian matrix $X$. 

\begin{theorem}\label{main1} 
Let $X$ be a $d\times d$ random matrix with jointly Gaussian entries and $\mathbb{E}X=0$, then
\[\mathbb{E}\|X\|\lesssim_{\epsilon}\|\mathbb{E}(X^{*}X)\|^{\frac{1}{2}}+\|\mathbb{E}(XX^{*})\|^{\frac{1}{2}}+d^{\epsilon}\|\mathbb{E}(X\otimes X)\|^{\frac{1}{2}},\]
for all $\epsilon>0$; here $\lesssim_\epsilon$ means less or equal up to a dimension-free multiplicative constant depending on $\epsilon$.
\end{theorem}

Note that $\mathbb{E}(X\otimes X)$ is a linear transformation on the $d^2$ dimensional inner product
space $M_{d}(\mathbb{R})$ of $d\times d$ real matrices with $\langle A,B\rangle=\mathrm{Tr}(AB^{*})$ for $A,B\in M_{d}(\mathbb{R})$.

Before presenting a range of guiding examples and discussing the sharpness of this inequality, we state a ``user-friendly'' version of it. One can see that the first statement of Theorem \ref{main2} is equivalent to Theorem \ref{main1} by taking the $A_{1},\ldots,A_{n}$ in Theorem \ref{main2} being certain appropriately scaled eigenvectors of $\mathbb{E}(X\otimes X)$ in Theorem \ref{main1}. Moreover, when all entries of $A_{1},\ldots,A_{n}$ are nonnegative, the $d^{\epsilon}$ factor can be replaced by $(\log d)^{2}$.

\begin{theorem}\label{main2}
Let $g_{1},\ldots,g_{n}$ be i.i.d.~standard Gaussian random variables and $A_{1},\ldots,A_{n}\in M_{d}(\mathbb{R})$ satisfy $\mathrm{Tr}(A_{k_{1}}A_{k_{2}}^{*})=0$ for all $k_{1}\neq k_{2}$ in $[n]$. Then
\[\mathbb{E}\left\|\sum_{k=1}^{n}g_{k}A_{k}\right\|\lesssim_{\epsilon}\left\|\sum_{k=1}^{n}A_{k}^{*}A_{k}\right\|^{\frac{1}{2}}+\left\|\sum_{k=1}^{n}A_{k}A_{k}^{*}\right\|^{\frac{1}{2}}+d^{\epsilon}\max_{k\in[n]}\|A_{k}\|_{F},\]
for all $\epsilon>0$. If moreover, all entries of $A_{1},\ldots,A_{n}$ are nonnegative, then
\[\mathbb{E}\left\|\sum_{k=1}^{n}g_{k}A_{k}\right\|\lesssim\left\|\sum_{k=1}^{n}A_{k}^{*}A_{k}\right\|^{\frac{1}{2}}+\left\|\sum_{k=1}^{n}A_{k}A_{k}^{*}\right\|^{\frac{1}{2}}+(\log d)^{2}\max_{k\in[n]}\|A_{k}\|_{F}.\]
\end{theorem}

While Theorem~\ref{main2} is the one we use in the guiding examples, it is worth formulating an inequality for Gaussian series without the orthogonality condition; the following follows immediately from Theorem~\ref{main1} by noticing that the Gaussian series is a Gaussian matrix.

\begin{corollary}\label{corollary:gaussianseries}
Let $g_{1},\ldots,g_{n}$ be i.i.d.~standard Gaussian random variables and $H_{1},\ldots,H_{n}\in M_{d}(\mathbb{R})$. Then
\[\mathbb{E}\left\|\sum_{k=1}^{n}g_{k}H_{k}\right\|\lesssim_{\epsilon}\left\|\sum_{k=1}^{n}H_{k}^{*}H_{k}\right\|^{\frac{1}{2}}+\left\|\sum_{k=1}^{n}H_{k}H_{k}^{*}\right\|^{\frac{1}{2}}+d^{\epsilon}\sup_{\substack{B\in M_{d}(\mathbb{R})\\\|B\|_{F}\leq 1}}\left(\sum_{k=1}^{n}\langle H_k,B\rangle^2\right)^{\frac{1}{2}}\]
\end{corollary}

\begin{remark}
 While outside the scope of this paper, we note that (i) Corollary \ref{corollary:gaussianseries} can be used to obtained non-asymptotic bounds on the expected spectral norm of sums of independent random matrices via the techniques described in~\cite{Troppelemappr} and (ii) it is, in general, possible to obtain tail bounds on the spectral norm of random matrices via a control on the expected spectral norm and a scalar concentration inequality.
\end{remark}

\begin{remark}
Theorem~\ref{main1} is not, in general, sharp. We expect that the $d^\epsilon$ factor is not needed and could be replaced by a $\sqrt{\log d}$ factor, but were not able to prove it. Furthermore, the term $\|\mathbb{E}(X \otimes X) \|$ does not appear to be the correct quantity in general. In particular, there are situations in which it is even weaker than the noncommutative Khintchine inequality \eqref{Introeq1}: namely, for $n=1$, we have $X=gA$ and $\|\mathbb{E}X\|\sim \|A\|$ while $\|\mathbb{E}(X \otimes X) \| = \|A\|_F^2$, which can be a factor of $d$ larger than $\|A\|^{2}$. Nevertheless, as we will see in the next section, Theorem~\ref{main1} captures the sharp behavior of the expected norm of a Gaussian matrix with correlated entries in several scenarios.
\end{remark}

\subsection{A Conjecture involving a weak variance parameter}

It has been conjectured, first implicitly in~\cite{TroppMatrixConcentration}, and then more explicitly in~\cite{Bandeira42Problems,Tropp2ndorder,vanHandel} %,Ramon-inhomogeneous2,BandeiraDing} 
that the correct parameter commanding the existence or not of the logarithmic factor in noncommutative Khintchine is the weak variance: for a $d\times d$ centered random matrix $X$ with jointly Gaussian entries and $\mathbb{E}X=0$, let
\[\sigma_{*}(X)=\sup_{\substack{v,w\in\mathbb{R}^{d}\\\|v\|_{2}=\|w\|_{2}=1}}(\mathbb{E}\langle Xv,w\rangle^{2})^{\frac{1}{2}}.\]
This parameter can be viewed as the injective norm of $\mathbb{E}X\otimes X$ when viewed as a fourth order
 tensor. It is also worth noting that this is the parameter governing fluctuations per Gaussian elimination
\[\mathbb{P}(|\|X\|-\mathbb{E}\|X\||\geq t)\leq 2e^{-t^{2}/(2\sigma_{*}(X)^{2})}.\]
Intuitively, in the language of Corollary~\ref{corollary:gaussianseries} and the particular case of self-adjoint matrices, the cancellations responsible for the removal of the $\sqrt{\log d}$ factor appear to be due to non-commutativity of the matrices $H_k$'s.

\begin{conjecture}\label{conj:NCK-improvement}
Let $X$ be a $d\times d$ centered random matrix with jointly Gaussian entries and, then
\[\mathbb{E}\|X\|\lesssim\|\mathbb{E}(X^{*}X)\|^{\frac{1}{2}}+\|\mathbb{E}(XX^{*})\|^{\frac{1}{2}} 
+\sqrt{\log d}\sup_{\substack{v,w\in\mathbb{R}^{d}\\\|v\|_{2}=\|w\|_{2}=1}}(\mathbb{E}\langle Xv,w\rangle^{2})^{\frac{1}{2}}\]
\end{conjecture}

We note that $\sigma_{*}(X)\leq\|\mathbb{E}(X\otimes X)\|^{\frac{1}{2}}$, since
\[\|\mathbb{E}(X\otimes X)\|=\sup_{\substack{B\in M_{d}(\mathbb{R})\\\|B\|_{F}\leq 1}}\mathbb{E}\langle X,B\rangle^{2}=\sup_{\substack{B\in M_{d}(\mathbb{R})\\\|B\|_{F}\leq 1}}\mathbb{E}[\mathrm{Tr}(XB^{*})]^{2}.\]
Also, the Cauchy-Schwarz inequality implies that $\sigma_{*}(X)\leq\|\mathbb{E}(X^{*}X)\|^{\frac{1}{2}}$.

Conjecture~\ref{conj:NCK-improvement} has been verified in the case of independent entries~\cite{BandeiravanHandel}. It is worth mentioning that when the matrix is very inhomogenous even the term $\sqrt{\log d}\,\sigma_\ast(X)$ may not be necessary~\cite{%Ramon-inhomogeneous1,
Ramon-inhomogeneous2}. There are two ways in which Theorem~\ref{main1} is weaker than Conjecture~\ref{conj:NCK-improvement}: (i) the dimensional factor is $d^\epsilon$ as opposed to $\sqrt{\log d}$; in the examples to be described, this limits the regimes in which our inequality is sharp; and (ii) the quantity $\|\mathbb{E}(X\otimes X)\|$ can in general be larger than $\sigma_{*}(X)$; it is worth mentioning however that the quantity $\sigma_{*}(X)$ in Conjecture~\ref{conj:NCK-improvement} appears to be difficult to compute, whereas $\|\mathbb{E}(X\otimes X)\|$ can be viewed as an easily computable (sometimes sharp) upper bound, at least in several cases in the next section. In Remark~\ref{remark:sigmaastvssigmaF} we highlight an interesting regime in which these two quantities are different and Conjecture~\ref{conj:NCK-improvement} would imply a stronger result.
%As it will be apparent from the proof of our main result, the cancellations responsible for the removal of the $\sqrt{\log d}$ factor exploit the non-commutativity of the matrices $A_k$ in the gaussian series.

%Moreover, one can always compute $\|\mathbb{E}(X\otimes X)\|$ in polynomial time whenever $\mathbb{E}(X\otimes X)$ is given.

%\begin{remark}
%Inhomogenous is another story.
%\end{remark}

\subsection*{Notation} 
Throughout this paper, if $T$ is a matrix or a linear transformation on an inner product space, $\|T\|$ denotes the spectral norm of $T$. The trace and the Frobenius norm of $T$ are denoted by $\mathrm{Tr}\,T$ and $\|T\|_{F}=\sqrt{\mathrm{Tr}(T^{*}T)}$, respectively. For $a,b>0$, we write $a\lesssim b$ when $a\leq Cb$ for some universal constant $C>0$; we write $a\lesssim_{\epsilon}b$ when $a\leq C_{\epsilon}b$ for some constant $C_{\epsilon}>0$ that depends only on $\epsilon$; we write $a\sim b$ when $a\lesssim b$ and $b\lesssim a$; we write $a\sim_{\epsilon}b$ when $a\lesssim_{\epsilon}b$ and $b\lesssim_{\epsilon}a$. For $n\in\mathbb{N}$, $[n]=\{1,\ldots,n\}$. For $d\in\mathbb{N}$, $(e_{1},\ldots,e_{d})$ is the canonical basis for $\mathbb{R}^{d}$.

\section{Guiding Examples and Applications}
\subsection{Gaussian on a subspace}

Consider the inner product space $M_{d}(\mathbb{R})$ of $d\times d$ real matrices with $\langle A,B\rangle=\mathrm{Tr}(AB^{*})$. Suppose that $\mathcal{M}$ is a subspace of $M_{d}(\mathbb{R})$ and $X$ is a standard Gaussian on $\mathcal{M}$, i.e., $X=\sum_{k=1}^{\mathrm{dim}\,\mathcal{M}}g_{k}A_{k}$, where $g_{1},\ldots,g_{\mathrm{dim}\,\mathcal{M}}$ are i.i.d.~standard Gaussian random variables and $(A_{1},\ldots,A_{\mathrm{dim}\,\mathcal{M}})$ is any orthonormal basis for $\mathcal{M}$. (The distribution of $X$ is independent of the choice of the orthonormal basis.) When $\mathrm{dim}\,\mathcal{M}=d$, the $\sqrt{\log d}$ factor in (\ref{Introeq1}) cannot always be removed, e.g., when $\mathcal{M}$ is the subspace of diagonal matrices. When $\mathrm{dim}\,\mathcal{M}=d^{2}$, we have $\mathcal{M}=M_{d}(\mathbb{R})$ so all the $d^{2}$ entries of $X$ are i.i.d.~standard Gaussian and the $\sqrt{\log d}$ factor can be removed.

In this paper, we show that, for any $\epsilon>0$, when $\mathrm{dim}\,\mathcal{M}\geq d^{1+\epsilon}$, the $\sqrt{\log d}$ factor can still be removed. Thus, there is a ``phase transition'' where the $\sqrt{\log d}$ factor cannot always be removed for $\mathrm{dim}\,\mathcal{M}=d$, but can be removed for $\mathrm{dim}\,\mathcal{M}\geq d^{1+\epsilon}$. Intuitively, this is because when all matrices in $\mathcal{M}$ are self-adjoint, it is possible that all matrices in $\mathcal{M}$ commute if $\mathrm{dim}\,\mathcal{M}=d$, but it is impossible that all matrices in $\mathcal{M}$ commute when $\mathrm{dim}\,\mathcal{M}>d$. As $\mathrm{dim}\,\mathcal{M}$ gets larger, the matrices in $\mathcal{M}$ are ``more noncommuting."

\begin{corollary}\label{subspaceopnorm}
If $X$ is a standard Gaussian on a subspace $\mathcal{M}$ of $M_{d}(\mathbb{R})$ and $\mathrm{dim}\,\mathcal{M}\geq d^{1+\epsilon}$ with $\epsilon>0$, then
\[\mathbb{E}\|X\|\sim_{\epsilon}\|\mathbb{E}(X^{*}X)\|^{\frac{1}{2}}+\|\mathbb{E}(XX^{*})\|^{\frac{1}{2}}.\]
\end{corollary}
\begin{proof}
Since $X$ is a standard Gaussian on $\mathcal{M}$, the expected Frobenius norm $\mathbb{E}\|X\|_{F}^{2}=\mathrm{dim}\,\mathcal{M}$ and the covariance $\mathbb{E}(X\otimes X)$ is the orthogonal projection from $M_{d}(\mathbb{R})$ onto $\mathcal{M}$. So the spectral norm $\|\mathbb{E}(X\otimes X)\|=1$. So $\|\mathbb{E}(X^{*}X)\|\geq\frac{1}{d}\mathbb{E}\mathrm{Tr}(X^{*}X)=\frac{1}{d}\mathbb{E}\|X\|_{F}^{2}=\frac{1}{d}\mathrm{dim}\,\mathcal{M}\geq d^{\epsilon}$. Thus, $d^{\frac{\epsilon}{2}}\|\mathbb{E}(X\otimes X)\|^{\frac{1}{2}}\leq\|\mathbb{E}(X^{*}X)\|^{\frac{1}{2}}$. The result follows from Theorem \ref{main1}.
\end{proof}

We expect the sharp condition to be $\mathrm{dim}\,\mathcal{M}\gtrsim d\log d$, but were not able to prove it. % as it would be predicted by Conjecture~\ref{conj:NCK-improvement}.

\subsection{Independent blocks}
%\begin{example}
In Theorem \ref{main2}, if we let $A_{1},\ldots,A_{d^{2}}$ be $A_{i,j}=b_{i,j}e_{i}e_{j}^{T}\in M_{d}(\mathbb{R})$ for $i,j\in[d]$, where $b_{i,j}>0$ for $i,j\in[d]$, then the second statement of Theorem \ref{main2} gives
\[\mathbb{E}\left\|\sum_{i,j\in[n]}g_{i,j}b_{i,j}e_{i}e_{j}^{T}\right\|\lesssim\max_{j\in[d]}\left(\sum_{i=1}^{d}|b_{i,j}|^{2}\right)^{\frac{1}{2}}+\max_{i\in[d]}\left(\sum_{j=1}^{d}|b_{i,j}|^{2}\right)^{\frac{1}{2}}+(\log d)^{2}\max_{i,j\in[n]}|b_{i,j}|,\]
where $(g_{i,j})_{i,j\in[d]}$ are i.i.d.~standard Gaussian random variables. This recovers a weaker version of a result by the first author and van Handel \cite{BandeiravanHandel}, who prove the estimate with the $(\log d)^{2}$ factor being replaced by $\sqrt{\log d}$, which is in fact, the optimal factor.
%\end{example}

A block version of this example better iluminates the difference between the weak variance and the quantity our inequality uses. We note this is different from the model of random lifts of graphs~\cite{Oliveira-lifts,Bordenave-Collins,BandeiraDing}.

\begin{corollary}\label{block}
For each $i,j\in[d]$, let $B_{i,j}$ be an $r\times r$ matrix and $g_{i,j}$ be independent standard Gaussian random variables. Consider the following $dr\times dr$ matrix
\[X=\begin{bmatrix}g_{1,1}B_{1,1}&\ldots&g_{1,d}B_{1,d}\\\vdots&\ddots&\vdots\\g_{d,1}B_{d,1}&\ldots&g_{d,d}B_{d,d}\end{bmatrix},\]
and $\displaystyle\gamma=\max_{j\in[d]}\left\|\sum_{i=1}^{d}B_{i,j}^{*}B_{i,j}\right\|^{\frac{1}{2}}+\max_{i\in[d]}\left\|\sum_{j=1}^{d}B_{i,j}B_{i,j}^{*}\right\|^{\frac{1}{2}}$. Then
\[\gamma\lesssim\mathbb{E}\|X\|\lesssim\gamma+(dr)^{\epsilon}\max_{i,j\in[d]}\|B_{i,j}\|_{F}.\]
If moreover, all entries of every $B_{i,j}$ are nonnegative, then
\[\gamma\lesssim\mathbb{E}\|X\|\lesssim\gamma+(\log(dr))^{2}\max_{i,j\in[d]}\|B_{i,j}\|_{F}.\]
\end{corollary}
\begin{proof}
This follows from Theorem \ref{main2} by taking $A_{1},\ldots,A_{d^{2}}\in M_{dr}(\mathbb{R})$ to be $A_{i,j}\in M_{dr}(\mathbb{R})$ being the matrix with the $(i,j)$-block being $B_{i,j}$ and the other blocks being $0$, where $i,j\in[d]$.
\end{proof}

\begin{remark}\label{remark:sigmaastvssigmaF}
We note that if Conjecture~\ref{conj:NCK-improvement} is true, then
\[\mathbb{E}\|X\|\lesssim\gamma+\sqrt{\log(dr)}\max_{i,j\in[d]}\|B_{i,j}\|,\]
where $\|B_{i,j}\|_{F}$ is replaced by $\|B_{i,j}\|$.
\end{remark}

\subsection{Indpendent rows}
\begin{corollary}\label{indeprow}
Suppose that $X$ is a $d_{1}\times d_{2}$ random matrix with independent rows and for $i\in[d_{1}]$, the $i$th row of $X$ is a centered Gaussian random vector with covariance matrix $B_{i}\in M_{d_{2}}(\mathbb{R})$. Then
\[\left\|\sum_{i=1}^{d_{1}}B_{i}\right\|^{\frac{1}{2}}+\max_{i\in[d_{1}]}[\mathrm{Tr}(B_{i})]^{\frac{1}{2}}\lesssim\mathbb{E}\|X\|\lesssim_{\epsilon}\left\|\sum_{i=1}^{d_{1}}B_{i}\right\|^{\frac{1}{2}}+\max_{i\in[d_{1}]}[\mathrm{Tr}(B_{i})]^{\frac{1}{2}}+\max(d_{1}^{\epsilon},d_{2}^{\epsilon})\max_{i\in[d_{1}]}\|B_{i}\|^{\frac{1}{2}},\]
for all $\epsilon>0$.
\end{corollary}
\begin{proof}
Write $X=\sum_{i=1}^{d_{1}}e_{i}x_{i}^{T}$ where each $x_{i}$ is a centered Gaussian random vector with covariance matrix $B_{i}\in M_{d_{2}}(\mathbb{R})$, and $x_{1},\ldots,x_{d_{1}}$ are independent. Thus, each $x_{i}$ can be written as $x_{i}=\sum_{j=1}^{d_{2}}g_{i,j}\sqrt{\lambda_{i,j}}v_{i,j}^{T}$, where $\lambda_{i,1},\ldots,\lambda_{i,d_{2}}$ are the eigenvalues of $B_{i}$ and $(v_{i,1},\ldots,v_{i,d_{2}})$ is an orthonormal basis for $\mathbb{R}^{d_{2}}$ consisting of the corresponding eigenvectors. Moreover, the $(g_{i,j})_{i\in[d_{1}],j\in[d_{2}]}$ are i.i.d.~standard Gaussian random variables.

We have $X=\sum_{i\in[d_{1}],j\in[d_{2}]}g_{i,j}\sqrt{\lambda_{i,j}}e_{i}v_{i,j}^{T}$. Let $A_{i,j}=\sqrt{\lambda_{i,j}}e_{i}v_{i,j}^{T}$ for $i\in[d_{1}],\,j\in[d_{2}]$. Note that $\mathrm{Tr}(A_{i_{1},j_{1}}A_{i_{2},j_{2}}^{*})=0$ whenever $(i_{1},j_{1})\neq(i_{2},j_{2})$. Thus, $X=\sum_{i\in[d_{1}],j\in[d_{2}]}g_{i,j}A_{i,j}$. Since
\[\sum_{i\in[d_{1}],j\in[d_{2}]}A_{i,j}^{*}A_{i,j}=\sum_{i\in[d_{1}],j\in[d_{2}]}\lambda_{i,j}v_{i,j}v_{i,j}^{T}=\sum_{i=1}^{d_{1}}B_{i},\]
\[\sum_{i\in[d_{1}],j\in[d_{2}]}A_{i,j}A_{i,j}^{*}=\sum_{i\in[d_{1}],j\in[d_{2}]}\lambda_{i,j}e_{i}e_{i}^{T}=\sum_{i=1}^{d_{1}}\left(\sum_{j=1}^{d_{2}}\lambda_{i,j}\right)e_{i}e_{i}^{T}=
\sum_{i=1}^{d_{1}}(\mathrm{Tr}\,B_{i})e_{i}e_{i}^{T},\]
\[\max_{i\in[d_{1}],j\in[d_{2}]}\|A_{i,j}\|_{F}=\max_{i\in[d_{1}],j\in[d_{2}]}\sqrt{\lambda_{i,j}}=\max_{i\in[d_{1}]}\|B_{i}\|,\]
by Theorem \ref{main2} and adding some zero rows/columns to each $A_{i,j}$ so that they become square matrices, the right hand side of the result follows. The left hand side is simply $\|\sum_{i\in[d_{1}],j\in[d_{2}]}A_{i,j}^{*}A_{i,j}\|^{\frac{1}{2}}+\|\sum_{i\in[d_{1}],j\in[d_{2}]}A_{i,j}A_{i,j}^{*}\|^{\frac{1}{2}}$.
\end{proof}
\begin{remark}
In Corollary \ref{indeprow}, if $\mathrm{Tr}(B_{i})\geq\max(d_{1}^{\epsilon},d_{2}^{\epsilon})\|B_{i}\|$ for all $i\in[d_{1}]$, or if each $B_{i}$ appears in $B_{1},\ldots,B_{d_{1}}$ at least $\max(d_{1}^{\epsilon},d_{2}^{\epsilon})$ times, then we obtain
\[\mathbb{E}\|X\|\sim_{\epsilon}\left\|\sum_{i=1}^{d_{1}}B_{i}\right\|^{\frac{1}{2}}+\max_{i\in[d_{1}]}[\mathrm{Tr}(B_{i})]^{\frac{1}{2}},\]
and so since $(\mathbb{E}\|X\|^{2})^{\frac{1}{2}}\lesssim\mathbb{E}\|X\|$ (by a Gaussian version of Kahane's inequality \cite{Kahane} or by concentration of $\|X\|$),
\[\mathbb{E}(\|X\|^{2})\sim_{\epsilon}\left\|\sum_{i=1}^{d_{1}}B_{i}\right\|+\max_{i\in[d_{1}]}\mathrm{Tr}(B_{i}).\]
\end{remark}

\subsection{Sample covariance}
\begin{corollary}\label{samplecov}
Suppose that $\mu$ is a probability measure on $\{B\in M_{d_{2}}(\mathbb{R})|\,B\text{ is positive semidefinite}\}$. Let $z_{1},\ldots,z_{d_{1}}$ be i.i.d.~random vectors in $\mathbb{R}^{d_{2}}$ chosen according to $\int\mathcal{N}(0,B)\,d\mu(B)$, i.e., $\mathbb{P}(z_{1}\in\mathcal{S})=\int\mathbb{P}(B^{\frac{1}{2}}g\in\mathcal{S})\,d\mu(B)$ for all measurable $\mathcal{S}\subset\mathbb{R}^{d_{2}}$, where $g$ is a standard Gaussian on $\mathbb{R}^{d_{2}}$. Let $Y=\sum_{i=1}^{d_{1}}z_{i}z_{i}^{T}\in M_{d_{2}}(\mathbb{R})$. If $\mathrm{Tr}(B)\geq\max(d_{1}^{\epsilon},d_{2}^{\epsilon})\|B\|$ $\mu$-a.s., then
\[\mathbb{E}\|Y\|\sim_{\epsilon}d_{1}\left\|\int B\,d\mu(B)\right\|+\mathbb{E}\max_{i\in[d_{1}]}\mathrm{Tr}(B_{i}),\]
where $B_{1},\ldots,B_{d_{1}}$ in $M_{d_{2}}(\mathbb{R})$ are i.i.d.~chosen according to $\mu$.
\end{corollary}
\begin{proof}
By assumption, $z_{1},\ldots,z_{d_{1}}$ are chosen as follows: first, choose i.i.d.~$B_{1},\ldots,B_{d_{1}}$ in $M_{d_{2}}(\mathbb{R})$ according to $\mu$ and then for each $i\in[d_{1}]$, take $z_{i}=B_{i}^{\frac{1}{2}}g_{i}$, where $g_{1},\ldots,g_{d_{1}}$ are i.i.d. standard Gaussian random variables. Let $X$ be the $d_{1}\times d_{2}$ matrix with the $i$th row of $X$ being $z_{i}$ for every $i\in[d_{1}]$. Note that $Y=X^{*}X$. Since $\mathrm{Tr}(B)\geq\max(d_{1}^{\epsilon},d_{2}^{\epsilon})\|B\|$ $\mu$-a.s., by Corollary \ref{indeprow} and the remark after Corollary \ref{indeprow}, conditioning on $B_{1},\ldots,B_{d_{1}}$, we have
\[\mathbb{E}(\|X\|^{2}\,|\,B_{1},\ldots,B_{d_{1}})\sim_{\epsilon}\left\|\sum_{i=1}^{d_{1}}B_{i}\right\|+\max_{i\in[d_{1}]}[\mathrm{Tr}(B_{i})].\]
Thus, since $Y=X^{*}X$,
\[\mathbb{E}\|Y\|\sim_{\epsilon}\mathbb{E}\left\|\sum_{i=1}^{d_{1}}B_{i}\right\|+\mathbb{E}\max_{i\in[d_{1}]}\mathrm{Tr}(B_{i}).\]
By \cite[Theorem 5.1(1)]{Troppelemappr},
\[\mathbb{E}\left\|\sum_{i=1}^{d_{1}}B_{i}\right\|\lesssim\left\|\sum_{i=1}^{d_{1}}\mathbb{E}B_{i}\right\|+(\log d_{2})\mathbb{E}\max_{i\in[d_{1}]}\|B_{i}\|.\]
But by assumption, $\mathrm{Tr}(B)\geq d_{2}^{\epsilon}\|B\|$ $\mu$-a.s. Therefore,
\[\mathbb{E}\|Y\|\sim_{\epsilon}\left\|\sum_{i=1}^{d_{1}}\mathbb{E}B_{i}\right\|+\mathbb{E}\max_{i\in[d_{1}]}\mathrm{Tr}(B_{i}).\]
Since $\displaystyle\mathbb{E}B_{i}=\int B\,d\mu(B)$ for all $i\in[d_{1}]$, the result follows.
\end{proof}
\begin{remark}
If the assumption $\mathrm{Tr}(B)\geq \max(d_{1}^{\epsilon},d_{2}^{\epsilon})\|B\|$ $\mu$-a.s.~is removed, Corollary \ref{samplecov} may fail. For example, take $d_{1}=d_{2}$ and $\mu$ to be the uniform probability measure over the subset $\{e_{1}e_{1}^{T},\ldots,e_{d_{2}}e_{d_{2}}^{T}\}$ of $M_{d_{2}}(\mathbb{R})$. Then $d_{1}\left\|\int B\,d\mu(B)\right\|+\mathbb{E}\max_{i\in[d_{1}]}\mathrm{Tr}(B_{i})\sim 1$, while $\mathbb{E}\|Y\|\geq\mathbb{E}\max_{i\in[d_{1}]}\|z_{i}\|_{2}^{2}\sim\log d_{1}$.
\end{remark}

\subsection{Glued entries}
\begin{corollary}\label{gridpartition}
Suppose that $\{S_{1},\ldots,S_{n}\}$ is a partition of $[d]\times[d]$ such that $|S_{1}|=\ldots=|S_{n}|\leq\frac{d}{(\log d)^{4}}$. Let $g_{1},\ldots,g_{n}$ be i.i.d.~standard Gaussian random variables. Consider the random matrix $X$ in $M_{d}(\mathbb{R})$ defined by $X_{i,j}=g_{k}$ for all $(i,j)\in S_{k}$ and $k\in[n]$. Then
\[\mathbb{E}\|X\|\sim\left\|\sum_{k=1}^{n}A_{k}^{*}A_{k}\right\|^{\frac{1}{2}}+\left\|\sum_{k=1}^{n}A_{k}A_{k}^{*}\right\|^{\frac{1}{2}},\]
where for $k\in[n]$, the matrix $A_{k}\in M_{d}(\mathbb{R})$ is defined by $(A_{k})_{i,j}=\begin{cases}1,&(i,j)\in S_{k}\\0,&\text{Otherwise}\end{cases}$.
\end{corollary}
\begin{proof}
Observe that $X=\sum_{k=1}^{n}g_{k}A_{k}$ and that $\mathrm{Tr}(A_{k_{1}}A_{k_{2}}^{*})=0$ for all $k_{1}\neq k_{2}$. Thus, by Theorem \ref{main2},
\[\mathbb{E}\|X\|\lesssim\left\|\sum_{k=1}^{n}A_{k}^{*}A_{k}\right\|^{\frac{1}{2}}+\left\|\sum_{k=1}^{n}A_{k}A_{k}^{*}\right\|^{\frac{1}{2}}+(\log d)^{2}\max_{k\in[n]}\|A_{k}\|_{F}.\]
Since $\|A_{k}\|_{F}^{2}=\mathrm{Tr}(A_{k}^{*}A_{k})=|S_{1}|$ for all $k\in[n]$,
\[\left\|\sum_{k=1}^{n}A_{k}^{*}A_{k}\right\|\geq\frac{1}{d}\mathrm{Tr}\left(\sum_{k=1}^{n}A_{k}^{*}A_{k}\right)=\frac{n}{d}|S_{1}|.\]
Thus, if $\sqrt{\frac{n}{d}}\geq(\log d)^{2}$, then $\|\sum_{k=1}^{n}A_{k}^{*}A_{k}\|^{\frac{1}{2}}\geq(\log d)^{2}\max_{k\in[n]}\|A_{k}\|_{F}$ and the result follows. To show that $\sqrt{\frac{n}{d}}\geq(\log d)^{2}$, note that $n|S_{1}|=\sum_{k=1}^{n}|S_{k}|=d^{2}$ so $\frac{n}{d}=\frac{d}{|S_{1}|}\geq(\log d)^{4}$ by assumption.
\end{proof}
\begin{remark}
When $|S_{1}|=1$, this result recovers the classical estimate for the spectral norm of a standard Gaussian matrix. When $|S_{1}|=d$, this result could fail. For example, take $S_{k}=\{(i,j)\in[d]\times[d]|\,i-j\equiv k\mod d\}$ for $k\in[d]$. Then $A_{k}=A_{1}^{k}$, for all $k\in[d]$, and $X=\sum_{k=1}^{d}g_{k}A_{k}$ is a random circulant matrix. We have $\|\sum_{k=1}^{d}g_{k}A_{k}\|=\|\sum_{k=1}^{d}g_{k}A_{1}^{k}\|=\sup_{w^{d}=1}|\sum_{k=1}^{d}g_{k}w^{k}|$ has expected value $\sim\sqrt{d\log d}$. On the other hand, since $A_{k}$ is a unitary for all $k\in[d]$, we have $\|\sum_{k=1}^{d}A_{k}^{*}A_{k}\|^{\frac{1}{2}}=\|\sum_{k=1}^{d}A_{k}A_{k}^{*}\|^{\frac{1}{2}}=\sqrt{d}$. Or if $X$ is a random self-adjoint Toeplitz matrix where in each row, the entries are i.i.d. standard Gaussian entries, then the $\sqrt{\log d}$ factor is also needed in this case, though $|S_{1}|,\ldots,|S_{d}|$ are all different~\cite{Meckes}.
\end{remark}

A particularly interesting case is when, for some $r>0$, the partition $\{S_{1},\ldots,S_{n}\}$ of $[d]\times[d]$ satisfies, for all $k\in[n]$,
\begin{enumerate}[(1)]
\item $|S_{k}|=r$;
\item $S_{k}$ has at most one entry in each row of $[d]\times[d]$;
\item $S_{k}$ has at most one entry in each column of $[d]\times[d]$.
\end{enumerate}
For $k\in[n]$, consider the matrix $A_{k}\in M_{d}(\mathbb{R})$ defined by $(A_{k})_{i,j}=\begin{cases}1,&(i,j)\in S_{k}\\0,&\text{Otherwise}\end{cases}$. We have
\[\left\|\sum_{k=1}^{n}A_{k}^{*}A_{k}\right\|^{\frac{1}{2}}=\left\|\sum_{k=1}^{n}A_{k}A_{k}^{*}\right\|^{\frac{1}{2}}=\sqrt{d}.\]
Indeed, $A_{k}^{*}A_{k}$ and $A_{k}A_{k}^{*}$ are diagonal matrices for all $k\in[n]$. For every $r\in[d]$, their $r$th diagonal entries are
\[\langle A_{k}^{*}A_{k}e_{r},e_{r}\rangle=\|A_{k}e_{r}\|_{2}^{2}=\begin{cases}1,&S_{k}\text{ has one entry in the }r\text{th column}\\0,&\text{Otherwise}\end{cases},\]
and
\[\langle A_{k}A_{k}^{*}e_{r},e_{r}\rangle=\|A_{k}^{*}e_{r}\|_{2}^{2}=\begin{cases}1,&S_{k}\text{ has one entry in the }r\text{th row}\\0,&\text{Otherwise}\end{cases}.\]
Since each row/column has $d$ entries and each entry belongs to exactly one $S_{k}$ (by assumption that $\{S_{1},\ldots,S_{n}\}$ is a partition), it follows that $\sum_{k=1}^{n}\langle A_{k}^{*}A_{k}e_{r},e_{r}\rangle=\sum_{k=1}^{n}\langle A_{k}A_{k}^{*}e_{r},e_{r}\rangle=d$ for every $r\in[d]$. So $\|\sum_{k=1}^{n}A_{k}^{*}A_{k}\|=\|\sum_{k=1}^{n}A_{k}A_{k}^{*}\|=d$.

In this case, if $r\leq\frac{d}{(\log d)^{4}}$, Corollary \ref{gridpartition} implies that
\begin{equation}\label{partialpermutationeq1}
\mathbb{E}\left\|\sum_{k=1}^{n}g_{k}A_{k}\right\|\sim\sqrt{d},
\end{equation}
where $g_{1},\ldots,g_{n}$ are i.i.d.~standard Gaussian random variables. Again, we expect this to hold for $r\leq\frac{d}{\log d}$ but were not able to prove it.

\section{Proof of the main theorem}
\subsection{Some estimations}
The first step is to prove Lemma \ref{tracecross2}, which is a result about real (random) matrices. However, it uses Lemma \ref{tropplemma}, which is over the complex, in an essential way. So the first two lemmas, which are needed to prove Lemma \ref{tracecross2}, involve both real and complex matrices. Let $M_{d}(\mathbb{C})$ be the space of all $d\times d$ complex matrices.
\begin{lemma}\label{orthtr}
If $\{B_{1},\ldots,B_{d^{2}}\}$ is an orthonormal basis for $M_{d}(\mathbb{R})$, i.e., $\mathrm{Tr}(B_{k_{1}}B_{k_{2}}^{*})=\begin{cases}1,&k_{1}=k_{2}\\0,&k_{1}\neq k_{2}\end{cases}$, then $\displaystyle\sum_{k=1}^{d^{2}}B_{k}^{*}LB_{k}=\mathrm{Tr}(L)I$ for all $L\in M_{d}(\mathbb{C})$.
\end{lemma}
\begin{proof}
Without loss of generality, $L\in M_{d}(\mathbb{R})$. Let $g_{1},\ldots,g_{d^{2}}$ be i.i.d.~standard Gaussian random variables. Since \[\sum_{k=1}^{d^{2}}B_{k}^{*}LB_{k}=\mathbb{E}\left(\sum_{k=1}^{d^{2}}g_{k}B_{k}\right)^{*}L\left(\sum_{k=1}^{d^{2}}g_{k}B_{k}\right)\]
and $\sum_{k=1}^{d^{2}}g_{k}B_{k}$ is independent of the choice of the orthonormal basis $\{B_{1},\ldots,B_{d^{2}}\}$, the matrix $\sum_{k=1}^{d^{2}}B_{k}^{*}LB_{k}$ is independent of the choice of the orthonormal basis $\{B_{1},\ldots,B_{d^{2}}\}$. We may take $\{B_{1},\ldots,B_{d^{2}}\}=\{e_{i}e_{j}^{T}|\,i,j\in[d]\}$. We have
\[\sum_{k=1}^{d^{2}}B_{k}^{*}LB_{k}=\sum_{i=1}^{d}\sum_{j=1}^{d}e_{i}e_{j}^{T}Le_{j}e_{i}^{T}=\left(\sum_{j=1}^{d}e_{j}^{T}Le_{j}\right)\sum_{i=1}^{d}e_{i}e_{i}^{T}=\mathrm{Tr}(L)I.\]
\end{proof}
\begin{lemma}\label{tracecross}
Suppose that $Q_{1},\ldots,Q_{5}\in M_{d}(\mathbb{C})$ are unitary, $Y\in M_{d}(\mathbb{R})$ is self-adjoint, $A_{1},\ldots,A_{n}\in M_{d}(\mathbb{R})$ are self-adjoint matrices and $\mathrm{Tr}(A_{k_{1}}A_{k_{2}})=0$ for all $k_{1}\neq k_{2}$ in $[n]$. Then
\[\left|\sum_{k_{1},k_{2}\in[n]}\mathrm{Tr}(Q_{1}Y^{2}Q_{2}A_{k_{1}}Q_{3}A_{k_{2}}Q_{4}A_{k_{1}}Q_{5}A_{k_{2}})\right|\leq\left(\max_{k\in[n]}\|A_{k}\|_{F}\right)^{2}\left\|\sum_{k=1}^{n}A_{k}^{2}\right\|\mathrm{Tr}(Y^{2}).\]
\end{lemma}
\begin{proof}
Without loss of generality, assume that $A_{k}\neq 0$ for all $k\in[n]$. Let $\beta=\max_{k\in[n]}\|A_{k}\|_{F}$. For each $k\in[n]$, let $\lambda_{k}=\|A_{k}\|_{F}$ and write $A_{k}=\lambda_{k}B_{k}$. Then $B_{1},\ldots,B_{n}$ are orthonormal in $M_{d}(\mathbb{R})$. Extend $B_{1},\ldots,B_{n}$ to an orthonormal basis $(B_{1},\ldots,B_{d^{2}})$ for $M_{d}(\mathbb{R})$. Note that $B_{n+1},\ldots,B_{d^{2}}$ are not necessarily self-adjoint. For a matrix $D\in M_{d}(\mathbb{C})$, define $|D|^{2}=D^{*}D$. We have
\begin{align*}
&\left|\sum_{k_{1},k_{2}\in[n]}\mathrm{Tr}(Q_{1}Y^{2}Q_{2}A_{k_{1}}Q_{3}A_{k_{2}}Q_{4}A_{k_{1}}Q_{5}A_{k_{2}})\right|\\=&
\left|\sum_{k_{1}=1}^{n}\mathrm{Tr}\left((YQ_{2}A_{k_{1}}Q_{3})\sum_{k_{2}=1}^{n}A_{k_{2}}Q_{4}A_{k_{1}}Q_{5}A_{k_{2}}Q_{1}Y\right)\right|\\\leq&
\sum_{k_{1}=1}^{n}\left[[\mathrm{Tr}(YQ_{2}A_{k_{1}}^{2}Q_{2}^{*}Y)]^{\frac{1}{2}}\left(\mathrm{Tr}\left|\sum_{k_{2}=1}^{n}A_{k_{2}}Q_{4}A_{k_{1}}Q_{5}A_{k_{2}}Q_{1}Y\right|^{2}\right)^{\frac{1}{2}}\right]\\\leq&
\left(\sum_{k_{1}=1}^{n}\mathrm{Tr}(YQ_{2}A_{k_{1}}^{2}Q_{2}^{*}Y)\right)^{\frac{1}{2}}\left(\sum_{k_{1}=1}^{n}\mathrm{Tr}\left|\sum_{k_{2}=1}^{n}A_{k_{2}}Q_{4}A_{k_{1}}Q_{5}A_{k_{2}}Q_{1}Y\right|^{2}\right)^{\frac{1}{2}}\\\leq&
\left(\sum_{k_{1}=1}^{n}\mathrm{Tr}(YQ_{2}A_{k_{1}}^{2}Q_{2}^{*}Y)\right)^{\frac{1}{2}}\left(\sum_{k_{1}=1}^{d^{2}}\beta^{2}\mathrm{Tr}\left|\sum_{k_{2}=1}^{n}A_{k_{2}}Q_{4}B_{k_{1}}Q_{5}A_{k_{2}}Q_{1}Y\right|^{2}\right)^{\frac{1}{2}},
\end{align*}
where we use the cyclic property of the trace in the first equality, we use Cauchy-Schwarz inequality in the first and second inequalities, and we use the fact that $A_{k}=\lambda_{k}B_{k}$ with $0\leq\lambda_{k}\leq\beta$ and extend the sum over $k_{1}$ to $1,\ldots,d^{2}$ in the last inequality.

For the first term,
\[\sum_{k_{1}=1}^{n}\mathrm{Tr}(YQ_{2}A_{k_{1}}^{2}Q_{2}^{*}Y)=\mathrm{Tr}\left(YQ_{2}\left(\sum_{k=1}^{n}A_{k}^{2}\right)Q_{2}^{*}Y\right)\leq\left\|\sum_{k=1}^{n}A_{k}^{2}\right\|\mathrm{Tr}(Y^{2}).\]
For the second term,
\begin{align*}
&\sum_{k_{1}=1}^{d^{2}}\beta^{2}\mathrm{Tr}\left|\sum_{k_{2}=1}^{n}A_{k_{2}}Q_{4}B_{k_{1}}Q_{5}A_{k_{2}}Q_{1}Y\right|^{2}\\=&
\beta^{2}\sum_{k_{1}=1}^{d^{2}}\mathrm{Tr}\left(\sum_{k_{3}=1}^{n}YQ_{1}^{*}A_{k_{3}}Q_{5}^{*}B_{k_{1}}^{*}Q_{4}^{*}A_{k_{3}}\right)
\left(\sum_{k_{2}=1}^{n}A_{k_{2}}Q_{4}B_{k_{1}}Q_{5}A_{k_{2}}Q_{1}Y\right)\\=&
\beta^{2}\sum_{k_{2},k_{3}\in[n]}\mathrm{Tr}\left(YQ_{1}^{*}A_{k_{3}}Q_{5}^{*}\left(\sum_{k_{1}=1}^{d^2}B_{k_{1}}^{*}Q_{4}^{*}A_{k_{3}}A_{k_{2}}Q_{4}B_{k_{1}}\right)
Q_{5}A_{k_{2}}Q_{1}Y\right)\\=&
\beta^{2}\sum_{k_{2},k_{3}\in[n]}\mathrm{Tr}(Q_{4}^{*}A_{k_{3}}A_{k_{2}}Q_{4})\mathrm{Tr}(YQ_{1}^{*}A_{k_{3}}Q_{5}^{*}Q_{5}A_{k_{2}}Q_{1}Y)\\=&
\beta^{2}\sum_{k_{2},k_{3}\in[n]}\mathrm{Tr}(A_{k_{3}}A_{k_{2}})\mathrm{Tr}(YQ_{1}^{*}A_{k_{3}}A_{k_{2}}Q_{1}Y)\\=&
\beta^{2}\sum_{k=1}^{n}\|A_{k}\|_{F}^{2}\mathrm{Tr}(YQ_{1}^{*}A_{k}^{2}Q_{1}Y)\\\leq&
\beta^{4}\sum_{k=1}^{n}\mathrm{Tr}(YQ_{1}^{*}A_{k}^{2}Q_{1}Y)=\beta^{4}\mathrm{Tr}\left(YQ_{1}^{*}\left(\sum_{k=1}^{n}A_{k}^{2}\right)Q_{1}Y\right)\leq\beta^{4}\left\|\sum_{k=1}^{n}A_{k}^{2}\right\|\mathrm{Tr}(Y^{2}),
\end{align*}
where we expand the $|\ldots|^{2}$ in the first equality, rearrange the sums in the second equality, use Lemma \ref{orthtr} in the third equality, use $Q_{4}Q_{4}^{*}=Q_{5}^{*}Q_{5}=I$ in the fourth equality, use $\mathrm{Tr}(A_{k_{1}}A_{k_{2}})=0$, for all $k_{1}\neq k_{2}$, in the fifth equality and use $\|A_{k}\|_{F}\leq\beta$ in the first inequality.

Therefore, the result follows.
\end{proof}
\begin{remark}
By modifying the proof of Lemma \ref{tracecross} slightly, one can see that if
$A_{1},\ldots,A_{n}\in M_{d}(\mathbb{R})$ are self-adjoint matrices and $\mathrm{Tr}(A_{k_{1}}A_{k_{2}})=0$ for all $k_{1}\neq k_{2}$ in $[n]$, then
\[\left|\sum_{k_{1},k_{2}\in[n]}\langle A_{k_{1}}Q_{1}A_{k_{2}}Q_{2}A_{k_{1}}Q_{3}A_{k_{2}}v,v\rangle\right|\leq\left(\max_{k\in[n]}\|A_{k}\|_{F}\right)^{2}\left\langle\sum_{k=1}^{n}A_{k}^{2}v,v\right\rangle,\]
for all $v\in\mathbb{R}^{d}$ and unitary $Q_{1},Q_{2},Q_{3}\in M_{d}(\mathbb{C})$. Thus, in this case, the quantity $w(\sum_{k=1}^{n}g_{k}A_{k})$, introduced in \cite{Tropp2ndorder}, satisfies
\[\sup_{Q_{1},Q_{2},Q_{3}}\left\|\sum_{k_{1},k_{2}\in[n]} A_{k_{1}}Q_{1}A_{k_{2}}Q_{2}A_{k_{1}}Q_{3}A_{k_{2}}\right\|^{\frac{1}{4}}\leq 2\left(\max_{k\in[n]}\|A_{k}\|_{F}\right)^{\frac{1}{2}}\left\|\sum_{k=1}^{n}A_{k}^{2}\right\|^{\frac{1}{4}}.\]

\end{remark}
\begin{lemma}[\cite{Tropp2ndorder}, Proposition 8.3]\label{tropplemma}
Suppose that $F:(M_{d}(\mathbb{C}))^{s}\to\mathbb{C}$ is a multilinear function and $X_{1},\ldots,X_{s}$ are random (not necessarily independent) self-adjoint matrices in $M_{d}(\mathbb{C})$ such that $\mathbb{E}\|X_{i}\|^{s}<\infty$ for all $i\in[s]$. Then
\[|\mathbb{E}F(X_{1},\ldots,X_{s})|\leq\max_{j\in[s]}\mathbb{E}\max_{Q_{1},\ldots,Q_{s}}|F(Q_{1},\ldots,Q_{j-1},Q_{j}X_{j}^{s},Q_{j+1},\ldots,Q_{s})|,\]
where the second maximum is over all $d\times d$ (random) unitary matrices $Q_{1},\ldots,Q_{s}$ in $M_{d}(\mathbb{C})$.
\end{lemma}
\begin{lemma}\label{tracecross2}
Suppose that $A_{1},\ldots,A_{n}\in M_{d}(\mathbb{R})$ are self-adjoint matrices and $\mathrm{Tr}(A_{k_{1}}A_{k_{2}})=0$ for all $k_{1}\neq k_{2}$ in $[n]$. Let $p_{1}\leq\ldots\leq p_{5}$ in $\mathbb{N}$ with $p_{5}$ being even and let $X_{1},\ldots,X_{p_{5}}$ be real random self-adjoint matrices such that $\mathbb{E}\|X_{i}\|^{p_{5}}<\infty$ for all $i\in[s]$. Then
\begin{align*}
&\left|\mathbb{E}\sum_{k_{1},k_{2}\in[n]}\mathrm{Tr}\left((\prod_{i=1}^{p_{1}}X_{i})A_{k_{1}}(\prod_{i=p_{1}+1}^{p_{2}}X_{i})A_{k_{2}}(\prod_{i=p_{2}+1}^{p_{3}}X_{i})A_{k_{1}}(\prod_{i=p_{3}+1}^{p_{4}}X_{i})A_{k_{2}}(\prod_{i=p_{4}+1}^{p_{5}}X_{i})\right)\right|\\\leq&
\left(\max_{k\in[n]}\|A_{k}\|_{F}\right)^{2}\left\|\sum_{k=1}^{n}A_{k}^{2}\right\|\max_{j\in[p_{5}]}\mathbb{E}\mathrm{Tr}(X_{j}^{p_{5}}),
\end{align*}
where empty products are the identity, e.g., $\prod_{i=p_{1}+1}^{p_{1}}X_{i}=I$. 
\end{lemma}
\begin{proof}
Define $F:(M_{d}(\mathbb{C}))^{p_{5}}\to\mathbb{C}$ by
\begin{align*}
&F(Y_{1},\ldots,Y_{p_{5}})\\=&
\sum_{k_{1},k_{2}\in[n]}\mathrm{Tr}\left((\prod_{i=1}^{p_{1}}Y_{i})A_{k_{1}}(\prod_{i=p_{1}+1}^{p_{2}}Y_{i})A_{k_{2}}(\prod_{i=p_{2}+1}^{p_{3}}Y_{i})A_{k_{1}}(\prod_{i=p_{3}+1}^{p_{4}}Y_{i})A_{k_{2}}(\prod_{i=p_{4}+1}^{p_{5}}Y_{i})\right).
\end{align*}
For all $j\in[p_{5}]$ and $d\times d$ unitary matrices $Q_{1},\ldots,Q_{p_{5}}$, there exist $d\times d$ unitary matrices $Q_{1}',\ldots,Q_{5}'$ such that
\[F(Q_{1},\ldots,Q_{j-1},Q_{j}X_{j}^{p_{5}},Q_{j+1},\ldots,Q_{p_{5}})=\sum_{k_{1},k_{2}\in[n]}\mathrm{Tr}(Q_{1}'X_{j}^{p_{5}}Q_{2}'A_{k_{1}}Q_{3}'A_{k_{2}}Q_{4}'A_{k_{1}}Q_{5}'A_{k_{2}}),\]
by the cyclic property of the trace. So by Lemma \ref{tropplemma},
\[|\mathbb{E}F(X_{1},\ldots,X_{p_{5}})|\leq\max_{j\in[p_{5}]}\mathbb{E}\max_{Q_{1}',\ldots,Q_{5}'}\left|\sum_{k_{1},k_{2}\in[n]}\mathrm{Tr}(Q_{1}'X_{j}^{p_{5}}Q_{2}'A_{k_{1}}Q_{3}'A_{k_{2}}Q_{4}'A_{k_{1}}Q_{5}'A_{k_{2}})\right|,\]
where the second maximum is over all $d\times d$ (random) unitary matrices $Q_{1}',\ldots,Q_{5}'$. Thus, since $p_{5}$ is even, by Lemma \ref{tracecross},
\[|\mathbb{E}F(X_{1},\ldots,X_{p_{5}})|\leq\left(\max_{k\in[n]}\|A_{k}\|_{F}\right)^{2}\left\|\sum_{k=1}^{n}A_{k}^{2}\right\|\max_{j\in[p_{5}]}\mathbb{E}\mathrm{Tr}(X_{j}^{p_{5}}).\]
\end{proof}
\subsection{Tensor products}\label{tensorsection}
Suppose that $S$ is a finite set. If $\nu$ is a partition of $S$ and $i,j\in S$, then $i\stackrel{\nu}{\sim}j$ means that $i$ and $j$ are in the same block of $\nu$. For partitions $\nu_{1}$ and $\nu_{2}$ of $S$, we write $\nu_{1}\leq\nu_{2}$ if whenever $i\stackrel{\nu_{1}}{\sim}j$, we have $i\stackrel{\nu_{2}}{\sim}j$. For example, $\{\{1\},\{2\},\{3,4\}\}\leq\{\{1,2\},\{3,4\}\}$. For a partition $\nu$ of $S$, a subset $S_{0}$ of $S$ {\it splits} $\nu$ if whenever $i\stackrel{\nu}{\sim}j$ and $j\in S_{0}$, we have $i\in S_{0}$, or equivalently, $S_{0}$ is a union of blocks of $\nu$. For a function $f:S\to T$, where $T$ is a set, we write $f\sim\nu$ if whenever $i\stackrel{\nu}{\sim}j$ in $S$, we have $f(i)=f(j)$, or equivalently, $f$ is constant on each block of $\nu$.

A {\it pair partition} of $S$ is a partition where each block has exactly two elements. The set of all pair partitions of $S$ is denoted by $\mathbb{P}_{2}(S)$. Note that if $|S|$ is odd then $\mathbb{P}_{2}(S)=\emptyset$.

Suppose that $S$ is totally ordered. A partition $\nu$ of $S$ is {\it noncrossing} if whenever $i_{1}<i_{2}<i_{3}<i_{4}$ in $S$ and $i_{1}\stackrel{\nu}{\sim}i_{3}$ and $i_{2}\stackrel{\nu}{\sim}i_{4}$, we have $i_{1}\stackrel{\nu}{\sim}i_{2}\stackrel{\nu}{\sim}i_{3}\stackrel{\nu}{\sim}i_{4}$. The set of all noncrossing pair partitions of $S$ is denoted by $\mathrm{NC}_{2}(S)$. A partition is {\it crossing} if it is not noncrossing. The set of all crossing pair partitions of $S$ is denoted by $\mathrm{Cr}_{2}(S)=\mathbb{P}_{2}(S)\backslash\mathrm{NC}_{2}(S)$.

In the following two lemmas, the tensor products are the usual multilinear tensor products for vector spaces.
\begin{lemma}\label{gaussiantensor}
Suppose that $V$ is a vector space over $\mathbb{R}$, $A_{1},\ldots,A_{n}\in V$ and $g_{1},\ldots,g_{n}$ are i.i.d.~standard Gaussian random variables. Let $X=\sum_{k=1}^{n}g_{k}A_{k}$ and $X^{\otimes p}=\underbrace{X\otimes\ldots\otimes X}_{p}$. Then
\[\mathbb{E}(X^{\otimes p})=\sum_{\nu\in\mathbb{P}_{2}([p])}\sum_{\substack{f:[p]\to[n]\\f\sim\nu}}A_{f(1)}\otimes\ldots\otimes A_{f(p)}.\]
\end{lemma}
\begin{proof}
If $p$ is odd, then both sides are $0$ by symmetry of $X$ and $\mathbb{P}_{2}([p])=\emptyset$. It is easy to see that the result holds for $p=2$. For an even number $p\geq 4$, by Gaussian integration by parts,
\[\mathbb{E}(X^{\otimes p})=\sum_{k=1}^{n}\mathbb{E}g_{k}A_{k}\otimes X^{\otimes p-1}=\sum_{k=1}^{n}\sum_{j=1}^{p-1}\mathbb{E}(A_{k}\otimes X^{\otimes(j-1)}\otimes A_{k}\otimes X^{\otimes(p-1-j)}),\]
where when $j=1$ or $p-1$, the term $X^{\otimes 0}$ is not present. So applying induction hypothesis to $\mathbb{E}(X^{\otimes(j-1)}\otimes X^{\otimes(p-1-j)})$, we obtain
\begin{eqnarray*}
\mathbb{E}(X^{\otimes p})&=&\sum_{j=1}^{p-1}\sum_{k=1}^{n}\sum_{\sigma\in\mathbb{P}_{2}(\{2,\ldots,j\}\cup\{j+2,\ldots,p\})}\sum_{\substack{h:\{2,\ldots,j\}\cup\{j+2,\ldots,p\}\to[n]\\f\sim\sigma}}\\&&
A_{k}\otimes(A_{h(2)}\otimes\ldots\otimes A_{h(j)})\otimes A_{k}\otimes(A_{h(j+2)}\otimes\ldots\otimes A_{h(p)})\\&=&
\sum_{\nu\in\mathbb{P}_{2}([p])}\sum_{\substack{f:[p]\to[n]\\f\sim\nu}}A_{f(1)}\otimes\ldots\otimes A_{f(p)},
\end{eqnarray*}
via the identification $\nu=\sigma\cup\{\{1,j+1\}\}$ and $f(i)=\begin{cases}k,&i=1\text{ or }j+1\\h(i),&\text{Otherwise}\end{cases}$.
\end{proof}
\begin{lemma}\label{gaussiantensor2}
Suppose that $V$ is a vector space over $\mathbb{R}$, $A_{1},\ldots,A_{n}\in V$ and $g_{1},\ldots,g_{n}$ are i.i.d.~standard Gaussian random variables. Let $\sigma$ be a partition of $[p]$. Then there exist random variables $X_{1},\ldots,X_{p}$ taking values in $V$ such that each individual $X_{i}$ has the same distribution over $V$ as $\sum_{k=1}^{n}g_{k}A_{k}$ and
\[\sum_{\substack{\nu\in\mathbb{P}_{2}([p])\\\nu\leq\sigma}}\sum_{\substack{f:[p]\to[n]\\f\sim\nu}}A_{f(1)}\otimes\ldots\otimes A_{f(p)}=\mathbb{E}(X_{1}\otimes\ldots\otimes X_{p}).\]
\end{lemma}
\begin{proof}
Without loss of generality, by permuting the order of the tensor product, we may assume that $\sigma$ is an interval partition of $[p]$. Write $\sigma=\{B_{1},\ldots,B_{r}\}$ in the ascending order. Each partition $\nu\in\mathbb{P}_{2}([p])$ with $\nu\leq\sigma$ corresponds to partitions $\nu_{1}\in\mathbb{P}_{2}(B_{1}),\ldots,\nu_{r}\in\mathbb{P}_{2}(B_{r})$, via the correspondence $\nu\mapsto(\nu|_{B_{1}},\ldots,\nu|_{B_{r}})$. Thus,
\begin{align*}
&\sum_{\substack{\nu\in\mathbb{P}_{2}([p])\\\nu\leq\sigma}}\sum_{\substack{f:[p]\to[n]\\f\sim\nu}}A_{f(1)}\otimes\ldots\otimes A_{f(p)}\\=&
\sum_{\nu_{1}\in\mathbb{P}_{2}(B_{1})}\ldots\sum_{\nu_{r}\in\mathbb{P}_{2}(B_{r})}\sum_{\substack{f_{1}:B_{1}\to[n]\\f_{1}\sim\nu_{1}}}\ldots\sum_{\substack{f_{r}:B_{r}\to[n]\\f_{r}\sim\nu_{r}}}\left(\bigotimes_{i\in B_{1}}A_{f_{1}(i)}\right)\otimes\ldots\otimes\left(\bigotimes_{i\in B_{r}}A_{f_{r}(i)}\right)\\=&
\left(\sum_{\nu_{1}\in\mathbb{P}_{2}(B_{1})}\sum_{\substack{f_{1}:B_{1}\to[n]\\f_{1}\sim\nu_{1}}}\bigotimes_{i\in B_{1}}A_{f_{1}(i)}\right)\otimes\ldots\otimes\left(\sum_{\nu_{r}\in\mathbb{P}_{2}(B_{r})}\sum_{\substack{f_{r}:B_{r}\to[n]\\f_{r}\sim\nu_{r}}}\bigotimes_{i\in B_{r}}A_{f_{r}(i)}\right),
\end{align*}
where $\otimes_{i\in B_{j}}$ is the tensor product in the ascending order of $B_{j}$; for example, if $B_{1}=\{1,2,3\}$ then $\otimes_{i\in B_{1}}A_{f_{1}(i)}=A_{f_{1}(1)}\otimes A_{f_{1}(2)}\otimes A_{f_{1}(3)}$. Suppose that $g_{k,j}$, for $k\in[n]$ and $j\in[r]$, are i.i.d.~Gaussian random variables. By Lemma \ref{gaussiantensor},
\[\mathbb{E}\left(\sum_{k=1}^{n}g_{k,j}A_{k}\right)^{\otimes|B_{j}|}=\sum_{\nu_{j}\in\mathbb{P}_{2}(B_{j})}\sum_{\substack{f_{j}:B_{j}\to[n]\\f_{j}\sim\nu_{j}}}\bigotimes_{i\in B_{j}}A_{f_{j}(i)},\]
for every $j\in[r]$. Therefore,
\begin{align*}
&\sum_{\substack{\nu\in\mathbb{P}_{2}([p])\\\nu\leq\sigma}}\sum_{\substack{f:[p]\to[n]\\f\sim\nu}}A_{f(1)}\otimes\ldots\otimes A_{f(p)}\\=&
\left[\mathbb{E}\left(\sum_{k=1}^{n}g_{k,1}A_{k}\right)^{\otimes|B_{1}|}\right]\otimes\ldots\otimes\left[\mathbb{E}\left(\sum_{k=1}^{n}g_{k,r}A_{k}\right)^{\otimes|B_{r}|}\right]\\=&
\mathbb{E}\left[\left(\sum_{k=1}^{n}g_{k,1}A_{k}\right)^{\otimes|B_{1}|}\otimes\ldots\otimes\left(\sum_{k=1}^{n}g_{k,r}A_{k}\right)^{\otimes|B_{r}|}\right],
\end{align*}
where the last equality follows from independence of the $g_{k,j}$. For each $j\in[r]$ and each $i\in B_{j}$, take $X_{i}=\sum_{k=1}^{n}g_{k,j}A_{k}$. (The $X_{i}$ is the same for all $i$ in the same block.) The result follows.
\end{proof}
\pagebreak

\subsection{Proof of the second statement of Theorem \ref{main2}}
\begin{lemma}[\cite{Buchholz}, Corollary 3]\label{buchholz}
Suppose that $A_{1},\ldots,A_{n}\in M_{d}(\mathbb{R})$ are self-adjoint matrices. Then
\[\left|\sum_{\substack{f:[p]\to[n]\\f\sim\nu}}\mathrm{Tr}(A_{f(1)}\ldots A_{f(p)})\right|\leq\mathrm{Tr}\left(\sum_{k=1}^{n}A_{k}^{2}\right)^{\frac{p}{2}},\]
for all even number $p\in\mathbb{N}$ and $\nu\in\mathbb{P}_{2}([p])$.
\end{lemma}
\begin{lemma}\label{nonnegativerecursion}
Suppose that $g_{1},\ldots,g_{n}$ are i.i.d.~standard Gaussian random variables, $A_{1},\ldots,A_{n}\in M_{d}(\mathbb{R})$ are self-adjoint with nonnegative entries and $\mathrm{Tr}(A_{k_{1}}A_{k_{2}})=0$ for all $k_{1}\neq k_{2}$ in $[n]$. Let $X=\sum_{k=1}^{n}g_{k}A_{k}$, where $g_{1},\ldots,g_{n}$ are i.i.d.~Gaussian random variables. Then
\[\mathbb{E}\mathrm{Tr}(X^{p})\leq 2^{p}\mathrm{Tr}\left(\sum_{k=1}^{n}A_{k}^{2}\right)^{\frac{p}{2}}+p^{4}\left(\max_{k\in[n]}\|A_{k}\|_{F}\right)^{2}\left\|\sum_{k=1}^{n}A_{k}^{2}\right\|\mathbb{E}\mathrm{Tr}(X^{p-4}),\]
for all even number $p\geq 4$.
\end{lemma}
\begin{proof}
By Lemma \ref{gaussiantensor},
\[\mathbb{E}(X^{p})=\sum_{\nu\in\mathbb{P}_{2}([p])}\sum_{\substack{f:[p]\to[n]\\f\sim\nu}}A_{f(1)}\ldots A_{f(p)},\]
so
\begin{align}\label{nonnegativerecursioneq1}
&\mathbb{E}\mathrm{Tr}(X^{p})\\=&
\sum_{\nu\in\mathbb{P}_{2}([p])}\sum_{\substack{f:[p]\to[n]\\f\sim\nu}}\mathrm{Tr}(A_{f(1)}\ldots A_{f(p)})\nonumber\\=&
\sum_{\nu\in\mathrm{NC}_{2}([p])}\sum_{\substack{f:[p]\to[n]\\f\sim\nu}}\mathrm{Tr}(A_{f(1)}\ldots A_{f(p)})+\sum_{\nu\in\mathrm{Cr}_{2}([p])}\sum_{\substack{f:[p]\to[n]\\f\sim\nu}}\mathrm{Tr}(A_{f(1)}\ldots A_{f(p)})\nonumber\\\leq&
2^{p}\mathrm{Tr}\left(\sum_{k=1}^{n}A_{k}^{2}\right)^{\frac{p}{2}}+\sum_{\nu\in\mathrm{Cr}_{2}([p])}\sum_{\substack{f:[p]\to[n]\\f\sim\nu}}\mathrm{Tr}(A_{f(1)}\ldots A_{f(p)}),\nonumber
\end{align}
where the last inequality follows from Lemma \ref{buchholz} and the fact that there at most $2^{p}$ noncrossing pair partitions of $[p]$. For every $\nu\in\mathrm{Cr}_{2}([p])$, there exist $i_{1}<i_{2}<i_{3}<i_{4}$ in $[p]$ such that $\{i_{1},i_{3}\},\{i_{2},i_{4}\}\in\nu$. So
\begin{equation}\label{nonnegativerecursioneq2}
\sum_{\nu\in\mathrm{Cr}_{2}([p])}\sum_{\substack{f:[p]\to[n]\\f\sim\nu}}\mathrm{Tr}(A_{f(1)}\ldots A_{f(p)})\leq
\sum_{i_{1}<\ldots<i_{4}\text{ in }[p]}\sum_{\substack{\nu\in\mathbb{P}_{2}([p])\\\{i_{1},i_{3}\},\{i_{2},i_{4}\}\in\nu}}\sum_{\substack{f:[p]\to[n]\\f\sim\nu}}\mathrm{Tr}(A_{f(1)}\ldots A_{f(p)}).
\end{equation}
Note that this is only an inequality since it involves some overcounting. We have also used the assumption that the entries of $A_{1},\ldots,A_{n}$ are nonnegative. Fix $i_{1}<i_{2}<i_{3}<i_{4}$ in $[p]$. We have
\begin{align*}
&\sum_{\substack{\nu\in\mathbb{P}_{2}([p])\\\{i_{1},i_{3}\},\{i_{2},i_{4}\}\in\nu}}\sum_{\substack{f:[p]\to[n]\\f\sim\nu}}\mathrm{Tr}(A_{f(1)}\ldots A_{f(p)})\\=&
\sum_{k_{1},k_{2}\in[n]}\sum_{\sigma\in\mathbb{P}_{2}([p]\backslash\{i_{1},\ldots,i_{4}\})}\sum_{\substack{f:[p]\backslash\{i_{1},\ldots,i_{4}\}\to[n]\\f\sim\sigma}}
\mathrm{Tr}(A_{f(1)}\ldots A_{f(i_{1}-1)}A_{k_{1}}A_{f(i_{1}+1)}\ldots A_{f(i_{2}-1)}A_{k_{2}}\\&
A_{f(i_{2}+1)}\ldots A_{f(i_{3}-1)}A_{k_{1}}A_{f(i_{3}+1)}\ldots A_{f(i_{4}-1)}A_{k_{2}}A_{f(i_{4}+1)}\ldots A_{f(p)}),
\end{align*}
via the identification $k_{1}=f(i_{1})=f(i_{3})$, $k_{2}=f(i_{2})=f(i_{4})$ and $\sigma=\nu|_{[p]\backslash\{i_{1},\ldots,i_{4}\}}$. Thus, by Lemma \ref{gaussiantensor},
\begin{align*}
&\sum_{\substack{\nu\in\mathbb{P}_{2}([p])\\\{i_{1},i_{3}\},\{i_{2},i_{4}\}\in\nu}}\sum_{\substack{f:[p]\to[n]\\f\sim\nu}}\mathrm{Tr}(A_{f(1)}\ldots A_{f(p)})\\=&
\sum_{k_{1},k_{2}\in[n]}\mathbb{E}\mathrm{Tr}(X^{i_{1}-1}A_{k_{1}}X^{i_{2}-i_{1}-1}A_{k_{2}}X^{i_{3}-i_{2}-1}A_{k_{1}}X^{i_{4}-i_{3}-1}A_{k_{2}}X^{p-i_{4}}).
\end{align*}
By Lemma \ref{tracecross2}, this is at most $\left(\max_{k\in[n]}\|A_{k}\|_{F}\right)^{2}\left\|\sum_{k=1}^{n}A_{k}^{2}\right\|\mathbb{E}\mathrm{Tr}(X^{p-4})$. Thus, by (\ref{nonnegativerecursioneq2}),
\[\sum_{\nu\in\mathrm{Cr}_{2}([p])}\sum_{\substack{f:[p]\to[n]\\f\sim\nu}}\mathrm{Tr}(A_{f(1)}\ldots A_{f(p)})\leq p^{4}\left(\max_{k\in[n]}\|A_{k}\|_{F}\right)^{2}\left\|\sum_{k=1}^{n}A_{k}^{2}\right\|\mathbb{E}\mathrm{Tr}(X^{p-4}).\]
By (\ref{nonnegativerecursioneq1}), the result follows.
\end{proof}
\begin{proof}[Proof of the second statement of Theorem \ref{main2}]
Without loss of generality, we may assume that $A_{1},\ldots,A_{n}$ are self-adjoint by replacing each $A_{k}$ by the self-adjoint matrix $\begin{bmatrix}0&A_{k}\\A_{k}^{*}&0\end{bmatrix}$. By Lemma \ref{nonnegativerecursion}, for all even number $4\leq p\leq\log d$,
\[\mathbb{E}\mathrm{Tr}(X^{p})\leq d\cdot 2^{p}\left\|\sum_{k=1}^{n}A_{k}^{2}\right\|^{\frac{p}{2}}+(\log d)^{4}\left(\max_{k\in[n]}\|A_{k}\|_{F}\right)^{2}\left\|\sum_{k=1}^{n}A_{k}^{2}\right\|\mathbb{E}\mathrm{Tr}(X^{p-4}).\]
Let $b_{1}=2\|\sum_{k=1}^{n}A_{k}^{2}\|^{\frac{1}{2}}$ and $b_{2}=(\log d)^{4}(\max_{k\in[n]}\|A_{k}\|_{F})^{2}\|\sum_{k=1}^{n}A_{k}^{2}\|$. For $p\leq\log d$, let $a_{p}=\mathbb{E}\mathrm{Tr}(X^{p})$. Then $a_{p}\leq d\cdot b_{1}^{p}+b_{2}a_{p-4}$, for all even number $4\leq p\leq\log d$, and $a_{0}=d$. Thus, for all $p\leq\log d$ with $p$ divisible by $4$, we have
\[a_{p}\leq d(b_{1}^{p}+b_{2}b_{1}^{p-4}+b_{2}^{2}b_{1}^{p-8}+\ldots+b_{2}^{\frac{p}{4}-1}b_{1}^{4}+b_{2}^{\frac{p}{4}}),\]
so by Young's inequality, $a_{p}\leq d(\frac{p}{4}+1)(b_{1}^{p}+b_{2}^{\frac{p}{4}})$. Since $\mathbb{E}\|X\|\leq(\mathbb{E}\mathrm{Tr}(X^{p}))^{\frac{1}{p}}=a_{p}^{\frac{1}{p}}$, taking $p$ to be the largest number divisible by $4$ and such that $p\leq\log d$, we obtain
\[\mathbb{E}\|X\|\lesssim b_{1}+b_{2}^{\frac{1}{4}}\lesssim\left\|\sum_{k=1}^{n}A_{k}^{2}\right\|^{\frac{1}{2}}+(\log d)(\max_{k\in[n]}\|A_{k}\|_{F})^{\frac{1}{2}}\left\|\sum_{k=1}^{n}A_{k}^{2}\right\|^{\frac{1}{4}}.\]
But $(\log d)(\max_{k\in[n]}\|A_{k}\|_{F})^{\frac{1}{2}}\|\sum_{k=1}^{n}A_{k}^{2}\|^{\frac{1}{4}}\leq\|\sum_{k=1}^{n}A_{k}^{2}\|^{\frac{1}{2}}+(\log d)^{2}\max_{k\in[n]}\|A_{k}\|_{F}$. Thus, the result follows.
\end{proof}
\subsection{Proof of the first statement of Theorem \ref{main2}}
Recall the notation at the beginning of Section \ref{tensorsection}
\begin{lemma}\label{combinatorics}
Assume that $p\in\mathbb{N}$ is even. There exists $\phi:\mathrm{Cr}_{2}([p])\to\{\mathrm{Partitions}~\mathrm{of}~[p]\}$ such that
\begin{enumerate}[(1)]
\item $\nu\leq\phi(\nu)$ for all $\nu\in\mathrm{Cr}_{2}([p])$;
\item whenever $\nu,\widehat{\nu}\in\mathrm{Cr}_{2}([p])$ satisfy $\widehat{\nu}\leq\phi(\nu)$, we have $\phi(\nu)=\phi(\widehat{\nu})$;
\item for every $\sigma\in\mathrm{ran}~\phi$, there exist $i_{1}<i_{2}<i_{3}<i_{4}$ in $[p]$ such that $\{i_{1},i_{3}\},\{i_{2},i_{4}\}\in\sigma$;
\item $\mathrm{ran}~\phi$ has at most $4^{p}p^{2}$ elements.
\end{enumerate}
\end{lemma}
\begin{proof}
For $\nu\in\mathrm{Cr}_{2}([p])$ and $k\in[p]$, let
\[S(\nu,k)=\{j\in[p]|\,j\stackrel{\nu}{\sim}i\text{ for some }i\in[k]\}.\]
Clearly $S(\nu,k)$ splits $\nu$ for all $k\in[p]$ and $\nu\in\mathrm{Cr}_{2}([p])$.

For every $\nu\in\mathrm{Cr}_{2}([p])$, let $k_{\nu}$ be the largest $k\in[p]$ for which $\nu|_{S(\nu,k)}$ is noncrossing. Take
\[\phi(\nu)=(\nu|_{S(\nu,k_{\nu}+1)})\cup\{[p]\backslash S(\nu,k_{\nu}+1)\},\]
for $\nu\in\mathrm{Cr}_{2}([p])$.

(1): Since $S(\nu,k_{\nu}+1)$ splits $\nu$, we have $\nu\leq\phi(\nu)$ for all $\nu\in\mathrm{Cr}_{2}([p])$. This proves (1).

(2): Suppose that $\nu,\widehat{\nu}\in\mathrm{Cr}_{2}([p])$ and $\widehat{\nu}\leq\phi(\nu)$. Then
\[\widehat{\nu}\leq(\nu|_{S(\nu,k_{\nu}+1)})\cup\{[p]\backslash S(\nu,k_{\nu}+1)\}\leq\{S(\nu,k_{\nu}+1),[p]\backslash S(\nu,k_{\nu}+1)\}.\]
Thus $S(\nu,k_{\nu}+1)$ splits $\widehat{\nu}$. Taking restriction to $S(\nu,k_{\nu}+1)$ in the first inequality, we obtain $\widehat{\nu}|_{S(\nu,k_{\nu}+1)}\leq\nu|_{S(\nu,k_{\nu}+1)}$. Since $S(\nu,k_{\nu}+1)$ splits each of $\nu$ and $\widehat{\nu}$ and since each of $\nu$ and $\widehat{\nu}$ are pair partitions, the restrictions $\nu|_{S(\nu,k_{\nu}+1)}$ and $\widehat{\nu}|_{S(\nu,k_{\nu}+1)}$ are still pair partitions. Thus, the only way $\widehat{\nu}|_{S(\nu,k_{\nu}+1)}\leq\nu|_{S(\nu,k_{\nu}+1)}$ can happen is when $\widehat{\nu}|_{S(\nu,k_{\nu}+1)}=\nu|_{S(\nu,k_{\nu}+1)}$. So we have $\widehat{\nu}|_{S(\nu,k_{\nu}+1)}=\nu|_{S(\nu,k_{\nu}+1)}$.

To show that $\phi(\widehat{\nu})=\phi(\nu)$, it remains to show that $S(\widehat{\nu},k_{\widehat{\nu}}+1)=S(\nu,k_{\nu}+1)$. First we show that
\begin{equation}\label{combinatoricseq1}
S(\widehat{\nu},k)=S(\nu,k)~\text{for all }k\in[k_{\nu}+1].
\end{equation}
Recall that we have proved that $S(\nu,k_{\nu}+1)$ splits $\widehat{\nu}$ and $\widehat{\nu}|_{S(\nu,k_{\nu}+1)}=\nu|_{S(\nu,k_{\nu}+1)}$. We will use repeatedly use these in the next few paragraphs.

Fix $k\in[k_{\nu}+1]$. If $j\in S(\widehat{\nu},k)$, i.e., $j\stackrel{\widehat{\nu}}{\sim}i$ for some $i\in[k]$, then $i\in[k_{\nu}+1]\subset S(\nu,k_{\nu}+1)$ so since $S(\nu,k_{\nu}+1)$ splits $\widehat{\nu}$, it follows that
$j\in S(\nu,k_{\nu}+1)$. Since $i,j\in S(\nu,k_{\nu}+1)$, $j\stackrel{\widehat{\nu}}{\sim}i$ and $\widehat{\nu}|_{S(\nu,k_{\nu}+1)}=\nu|_{S(\nu,k_{\nu}+1)}$, we have $j\stackrel{\nu}{\sim}i$. So $j\in S(\nu,k)$. Thus, $S(\widehat{\nu},k)\subset S(\nu,k)$.

Conversely, if $j\in S(\nu,k)$ then $j\stackrel{\nu}{\sim}i$ for some $i\in[k]$. Thus $i\in[k_{\nu}+1]$ so $j\in S(\nu,k_{\nu}+1)$ by definition of $S(\nu,k_{\nu}+1)$. Since $i,j\in S(\nu,k_{\nu}+1)$, $j\stackrel{\nu}{\sim}i$ and $\widehat{\nu}|_{S(\nu,k_{\nu}+1)}=\nu|_{S(\nu,k_{\nu}+1)}$, we have $j\stackrel{\widehat{\nu}}{\sim}i$. So $j\in S(\widehat{\nu},k)$. Thus, $S(\nu,k)\subset S(\widehat{\nu},k)$. This proves (\ref{combinatoricseq1}).

Since $\widehat{\nu}|_{S(\nu,k_{\nu}+1)}=\nu|_{S(\nu,k_{\nu}+1)}$, we have $\widehat{\nu}|_{S(\nu,k)}=\nu|_{S(\nu,k)}$, for all $k\in[k_{\nu}+1]$, since $S(\nu,k)\subset S(\nu,k_{\nu}+1)$. So by (\ref{combinatoricseq1}), $\widehat{\nu}|_{S(\widehat{\nu},k)}=\nu|_{S(\nu,k)}$, for all $k\in[k_{\nu}+1]$, where the restriction on the left hand side becomes $S(\widehat{\nu},k)$. Thus, by definition of $k_{\nu}$, we have that $\widehat{\nu}|_{S(\widehat{\nu},k_{\nu})}=\nu|_{S(\nu,k_{\nu})}$ is noncrossing and $\widehat{\nu}|_{S(\widehat{\nu},k_{\nu}+1)}=\nu|_{S(\nu,k_{\nu}+1)}$ is crossing. So by definition of $k_{\widehat{\nu}}$, we have $k_{\widehat{\nu}}=k_{\nu}$. So \[S(\widehat{\nu},k_{\widehat{\nu}}+1)=S(\widehat{\nu},k_{\nu}+1)=S(\nu,k_{\nu}+1),\]
by (\ref{combinatoricseq1}). This proves the remaining thing needed to obtain $\phi(\widehat{\nu})=\phi(\nu)$ as mentioned above.

(3): Let $\nu\in\mathrm{Cr}_{2}([p])$. By definition of $k_{\nu}$, the partition $\nu|_{S(\nu,k_{\nu}+1)}$ is crossing. Since $S(\nu,k_{\nu}+1)$ splits $\nu$ and $\nu$ is a pair partition, $\nu|_{S(\nu,k_{\nu}+1)}$ is still a pair partition. Thus, $\nu|_{S(\nu,k_{\nu}+1)}$ is a crossing pair partition. So there exist $i_{1}<i_{2}<i_{3}<i_{4}$ in $[p]$ such that $\{i_{1},i_{3}\},\{i_{2},i_{4}\}\in\nu|_{S(\nu,k_{\nu}+1)}$. So $\{i_{1},i_{3}\},\{i_{2},i_{4}\}\in\phi(\nu)$. This proves (3).

(4): For every $\nu\in\mathrm{Cr}_{2}([p])$,
\[\phi(\nu)=(\nu|_{S(\nu,k_{\nu})})\cup(\nu|_{S(\nu,k_{\nu}+1)\backslash S(\nu,k_{\nu})})\cup\{[p]\backslash S(\nu,k_{\nu}+1)\}.\]
Since $\nu$ is a pair partition, $S(\nu,k+1)\backslash S(\nu,k)$ has at most $2$ elements for every $k\in[p]$, namely, $k+1$ and another one in the same $\nu$-block as $k+1$.

There are at most $2^{p}$ sets of the form $S(\nu,k_{\nu})$ for some $\nu\in\mathrm{Cr}_{2}([p])$.\\
For each fixed $S(\nu,k_{\nu})$, there are at most $2^{p}$ possible noncrossing pair partitions $\nu|_{S(\nu,k_{\nu})}$.\\
There are at most $p^{2}$ choices of $S(\nu,k_{\nu}+1)\backslash S(\nu,k_{\nu})$ and\\
with $S(\nu,k_{\nu}+1)\backslash S(\nu,k_{\nu})$ being fixed, there is only choice of $\nu|_{S(\nu,k_{\nu}+1)\backslash S(\nu,k_{\nu})}$.

Therefore, there are at most $2^{p}\cdot 2^{p}\cdot p^{2}$ partitions of the form $\phi(\nu)$ for some $\nu\in\mathrm{Cr}_{2}([p])$.
\end{proof}
\begin{lemma}\label{combinatorics2}
Assume that $p\in\mathbb{N}$ is even. There exist partitions $\nu_{1},\ldots,\nu_{q}$ of $[p]$ such that
\begin{enumerate}[(1)]
\item every $\nu\in\mathrm{Cr}_{2}([p])$ is in exactly one of the sets\\
$\{\nu\in\mathrm{Cr}_{2}([p])|\,\nu\leq\nu_{1}\},\ldots,\{\nu\in\mathrm{Cr}_{2}([p])|\,\nu\leq\nu_{q}\}$.
\item for every $t\in[q]$, there exist $i_{1}<i_{2}<i_{3}<i_{4}$ in $[p]$ such that $\{i_{1},i_{3}\},\{i_{2},i_{4}\}\in\nu_{t}$.
\item $q\leq 4^{p}p^{2}$.
\end{enumerate}
\end{lemma}
\begin{proof}
This follows from Lemma \ref{combinatorics} by enumerating the range of $\phi$ as $\nu_{1},\ldots,\nu_{q}$.
\end{proof}
\begin{lemma}\label{recursion}
Suppose that $g_{1},\ldots,g_{n}$ are i.i.d.~standard Gaussian random variables, $A_{1},\ldots,A_{n}\in M_{d}(\mathbb{R})$ are self-adjoint and $\mathrm{Tr}(A_{k_{1}}A_{k_{2}})=0$ for all $k_{1}\neq k_{2}$ in $[n]$. Let $X=\sum_{k=1}^{n}g_{k}A_{k}$, where $g_{1},\ldots,g_{n}$ are i.i.d.~Gaussian random variables. Then
\[\mathbb{E}\mathrm{Tr}(X^{p})\leq 2^{p}\mathrm{Tr}\left(\sum_{k=1}^{n}A_{k}^{2}\right)^{\frac{p}{2}}+8^{p}\left(\max_{i\in[n]}\|A_{i}\|_{F}\right)^{2}\left\|\sum_{k=1}^{n}A_{k}^{2}\right\|\mathbb{E}\mathrm{Tr}(X^{p-4}),\]
for all even number $p\geq 4$.
\end{lemma}
\begin{proof}
By Lemma \ref{gaussiantensor},
\[\mathbb{E}(X^{p})=\sum_{\nu\in\mathbb{P}_{2}([p])}\sum_{\substack{f:[p]\to[n]\\f\sim\nu}}A_{f(1)}\ldots A_{f(p)},\]
so
\begin{align}\label{recursioneq1}
&\mathbb{E}\mathrm{Tr}(X^{p})\\=&
\sum_{\nu\in\mathbb{P}_{2}([p])}\sum_{\substack{f:[p]\to[n]\\f\sim\nu}}\mathrm{Tr}(A_{f(1)}\ldots A_{f(p)})\nonumber\\\leq&
\left|\sum_{\nu\in\mathrm{NC}_{2}([p])}\sum_{\substack{f:[p]\to[n]\\f\sim\nu}}\mathrm{Tr}(A_{f(1)}\ldots A_{f(p)})\right|+\left|\sum_{\nu\in\mathrm{Cr}_{2}([p])}\sum_{\substack{f:[p]\to[n]\\f\sim\nu}}\mathrm{Tr}(A_{f(1)}\ldots A_{f(p)})\right|\nonumber\\\leq&
2^{p}\mathrm{Tr}\left(\sum_{k=1}^{n}A_{k}^{2}\right)^{\frac{p}{2}}+\left|\sum_{\nu\in\mathrm{Cr}_{2}([p])}\sum_{\substack{f:[p]\to[n]\\f\sim\nu}}\mathrm{Tr}(A_{f(1)}\ldots A_{f(p)})\right|,\nonumber
\end{align}
where the last inequality follows from Lemma \ref{buchholz} and the fact that there at most $2^{p}$ noncrossing pair partitions of $[p]$. Let $\nu_{1},\ldots,\nu_{q}$ be obtained from Lemma \ref{combinatorics2} with $q\leq 4^{p}p^{2}$. Since every $\nu\in\mathrm{Cr}_{2}([p])$ satisfies $\nu\leq\nu_{t}$ for exactly one $t\in[q]$,
\begin{equation}\label{recursioneq2}
\sum_{\nu\in\mathrm{Cr}_{2}([p])}\sum_{\substack{f:[p]\to[n]\\f\sim\nu}}\mathrm{Tr}(A_{f(1)}\ldots A_{f(p)})=\sum_{t=1}^{q}\sum_{\substack{\nu\in\mathrm{Cr}_{2}([p])\\\nu\leq\nu_{t}}}\sum_{\substack{f:[p]\to[n]\\f\sim\nu}}\mathrm{Tr}(A_{f(1)}\ldots A_{f(p)}).
\end{equation}
Fix $t\in[q]$. By the properties of $\nu_{t}$ from Lemma \ref{combinatorics2}, there exist $i_{1}<i_{2}<i_{3}<i_{4}$ in $[p]$ such that $\{i_{1},i_{3}\},\{i_{2},i_{4}\}\in\nu_{t}$. For every $\nu\in\mathrm{Cr}_{2}([p])$ such that $\nu\leq\nu_{t}$, since $\{i_{1},i_{3}\},\{i_{2},i_{4}\}\in\nu_{t}$ and $\nu$ is a pair partition, we have $\{i_{1},i_{3}\},\{i_{2},i_{4}\}\in\nu$. Let $\omega_{t}=\nu_{t}|_{[p]\backslash\{i_{1},\ldots,i_{4}\}}$. We have
\begin{align*}
&\sum_{\substack{\nu\in\mathrm{Cr}_{2}([p])\\\nu\leq\nu_{t}}}\sum_{\substack{f:[p]\to[n]\\f\sim\nu}}\mathrm{Tr}(A_{f(1)}\ldots A_{f(p)})\\=&
\sum_{k_{1},k_{2}\in[n]}\sum_{\substack{\sigma\in\mathbb{P}_{2}([p]\backslash\{i_{1},\ldots,i_{4}\})\\\sigma\leq\omega_{t}}}\sum_{\substack{f:[p]\backslash\{i_{1},\ldots,i_{4}\}\to[n]\\f\sim\sigma}}
\mathrm{Tr}(A_{f(1)}\ldots A_{f(i_{1}-1)}A_{k_{1}}A_{f(i_{1}+1)}\ldots A_{f(i_{2}-1)}A_{k_{2}}\\&
A_{f(i_{2}+1)}\ldots A_{f(i_{3}-1)}A_{k_{1}}A_{f(i_{3}+1)}\ldots A_{f(i_{4}-1)}A_{k_{2}}A_{f(i_{4}+1)}\ldots A_{f(p)}),
\end{align*}
via the identification $k_{1}=f(i_{1})=f(i_{3})$, $k_{2}=f(i_{2})=f(i_{4})$ and $\sigma=\nu|_{[p]\backslash\{i_{1},\ldots,i_{4}\}}$. Thus, by Lemma \ref{gaussiantensor2},
\begin{align}\label{recursioneq3}
&\sum_{\substack{\nu\in\mathbb{P}_{2}([p])\\\nu\leq\nu_{t}}}\sum_{\substack{f:[p]\to[n]\\f\sim\nu}}\mathrm{Tr}(A_{f(1)}\ldots A_{f(p)})\\=&
\sum_{k_{1},k_{2}\in[n]}\mathbb{E}\mathrm{Tr}\left((\prod_{i=1}^{i_{1}-1}X_{i})A_{k_{1}}(\prod_{i=i_{1}+1}^{i_{2}-1}X_{i})A_{k_{2}}(\prod_{i=i_{2}+1}^{i_{3}-3}X_{i})A_{k_{1}}(\prod_{i=i_{3}+1}^{i_{4}}X_{i})A_{k_{2}}(\prod_{i=i_{4}+1}^{p}X_{i})\right),\nonumber
\end{align}
for some random matrices $X_{1},\ldots,X_{p}$ (with $X_{i_{1}},X_{i_{2}},X_{i_{3}},X_{i_{4}}$ skipped) in $M_{d}(\mathbb{R})$ such that each individual $X_{i}$ has the same distribution as $X=\sum_{k=1}^{n}g_{k}A_{k}$. By Lemma \ref{tracecross2}, the absolute value of the expression (\ref{recursioneq3}) is at most $\left(\max_{k\in[n]}\|A_{k}\|_{F}\right)^{2}\left\|\sum_{k=1}^{n}A_{k}^{2}\right\|\max_{j}\mathbb{E}\mathrm{Tr}(X_{j}^{p-4})=\left(\max_{k\in[n]}\|A_{k}\|_{F}\right)^{2}\left\|\sum_{k=1}^{n}A_{k}^{2}\right\|\mathbb{E}\mathrm{Tr}(X^{p-4})$, since each $X_{j}$ has the same distribution as $X$. Thus, by (\ref{recursioneq2}),
\[\sum_{\nu\in\mathrm{Cr}_{2}([p])}\sum_{\substack{f:[p]\to[n]\\f\sim\nu}}\mathrm{Tr}(A_{f(1)}\ldots A_{f(p)})\leq q\left(\max_{k\in[n]}\|A_{k}\|_{F}\right)^{2}\left\|\sum_{k=1}^{n}A_{k}^{2}\right\|\mathbb{E}\mathrm{Tr}(X^{p-4}).\]
Since $q\leq 4^{p}p^{2}\leq 8^{p}$, by (\ref{recursioneq1}), the result follows.
\end{proof}
\begin{proof}[Proof of the first statement of Theorem \ref{main2}]
Without loss of generality, we may assume that $A_{1},\ldots,A_{n}$ are self-adjoint by replacing each $A_{k}$ by the self-adjoint matrix $\begin{bmatrix}0&A_{k}\\A_{k}^{*}&0\end{bmatrix}$. Fix $\epsilon>0$. By Lemma \ref{recursion}, for all even number $4\leq p\leq\epsilon\log_{8}d$,
\[\mathbb{E}\mathrm{Tr}(X^{p})\leq d\cdot 2^{p}\left\|\sum_{k=1}^{n}A_{k}^{2}\right\|^{\frac{p}{2}}+d^{\epsilon}\left(\max_{k\in[n]}\|A_{k}\|_{F}\right)^{2}\left\|\sum_{k=1}^{n}A_{k}^{2}\right\|\mathbb{E}\mathrm{Tr}(X^{p-4}).\]
Let $b_{1}=2\|\sum_{k=1}^{n}A_{k}^{2}\|^{\frac{1}{2}}$ and $b_{2}=d^{\epsilon}(\max_{i\in[n]}\|A_{i}\|_{F})^{2}\|\sum_{k=1}^{n}A_{k}^{2}\|$. For $p\leq\epsilon\log_{8}d$, let $a_{p}=\mathbb{E}\mathrm{Tr}(X^{p})$. Then $a_{p}\leq d\cdot b_{1}^{p}+b_{2}a_{p-4}$, for all even number $4\leq p\leq\epsilon\log_{8}d$, and $a_{0}=d$. Thus, for all $p\leq\epsilon\log_{8}d$ with $p$ divisible by $4$, we have
\[a_{p}\leq d(b_{1}^{p}+b_{2}b_{1}^{p-4}+b_{2}^{2}b_{1}^{p-8}+\ldots+b_{2}^{\frac{p}{4}-1}b_{1}^{4}+b_{2}^{\frac{p}{4}}),\]
so by Young's inequality, $a_{p}\leq d(\frac{p}{4}+1)(b_{1}^{p}+b_{2}^{\frac{p}{4}})$. Since $\mathbb{E}\|X\|\leq(\mathbb{E}\mathrm{Tr}(X^{p}))^{\frac{1}{p}}=a_{p}^{\frac{1}{p}}$, taking $p$ to be the largest number divisible by $4$ and such that $p\leq\epsilon\log_{8}d$, we obtain
\[\mathbb{E}\|X\|\lesssim_{\epsilon}b_{1}+b_{2}^{\frac{1}{4}}\lesssim\left\|\sum_{k=1}^{n}A_{k}^{2}\right\|^{\frac{1}{2}}+d^{\frac{\epsilon}{4}}(\max_{i\in[n]}\|A_{i}\|_{F})^{\frac{1}{2}}\left\|\sum_{k=1}^{n}A_{k}^{2}\right\|^{\frac{1}{4}}.\]
But $d^{\frac{\epsilon}{4}}(\max_{i\in[n]}\|A_{i}\|_{F})^{\frac{1}{2}}\|\sum_{k=1}^{n}A_{k}^{2}\|^{\frac{1}{4}}\leq\|\sum_{k=1}^{n}A_{k}^{2}\|^{\frac{1}{2}}+d^{\frac{\epsilon}{2}}\max_{i\in[n]}\|A_{i}\|_{F}$. Thus, the result follows.
\end{proof}
\section{Sample covariance matrix}
\begin{theorem}\label{samplecovthm}
Suppose that $X_{1},\ldots,X_{M}$ are $d\times d$ independent random matrices and for each $r\in[M]$, the entries of $X_{r}$ are jointly Gaussian entries and $\mathbb{E}X_{r}=0$. Then
\[\mathbb{E}\left\|\sum_{r=1}^{M}(X_{r}^{*}X_{r}-\mathbb{E}(X_{r}^{*}X_{r}))\right\|\lesssim_{\epsilon}\left(\sum_{r=1}^{M}\left(\|\mathbb{E}(X_{r}^{*}X_{r})\|^{2}+\|\mathbb{E}(X_{r}X_{r}^{*})\|^{2}+d^{\epsilon}\|\mathbb{E}(X_{r}\otimes X_{r})\|^{2}\right)\right)^{\frac{1}{2}},\]
for all $\epsilon>0$.
\end{theorem}
\begin{proof}
For each $r\in[M]$, let $\widetilde{X}_{r}$ be an independent copy of $X_{r}$ so that $X_{1},\ldots,X_{M},\widetilde{X}_{1},\ldots,\widetilde{X}_{M}$ are independent. Then
\begin{equation}\label{samplecovthmeq1}
\mathbb{E}\left\|\sum_{r=1}^{M}(X_{r}^{*}X_{r}-\mathbb{E}(X_{r}^{*}X_{r}))\right\|\leq\mathbb{E}\left\|\sum_{r=1}^{M}(X_{r}^{*}X_{r}-\widetilde{X}_{r}^{*}\widetilde{X}_{r})\right\|.
\end{equation}
For every $r\in[M]$,
\begin{align*}
&X_{r}X_{r}^{*}-\widetilde{X}_{r}^{*}\widetilde{X}_{r}\\=&
\int_{0}^{\frac{\pi}{2}}\frac{d}{d\theta}[(\widetilde{X}_{r}\cos\theta+X_{r}\sin\theta)^{*}(\widetilde{X}_{r}\cos\theta+X_{r}\sin\theta))]\,d\theta\\=&
\int_{0}^{\frac{\pi}{2}}(-\widetilde{X}_{r}\sin\theta+X_{r}\cos\theta)^{*}(\widetilde{X}_{r}\cos\theta+X_{r}\sin\theta)\,d\theta\\&
+\int_{0}^{\frac{\pi}{2}}(\widetilde{X}_{r}\cos\theta+X_{r}\sin\theta)^{*}(-\widetilde{X}_{r}\sin\theta+X_{r}\cos\theta)\,d\theta,
\end{align*}
and since $X_{r},\widetilde{X}_{r}$ are independent and have the same centered Gaussian distribution on $M_{d}(\mathbb{R})$, the pair of random matrices $(-\widetilde{X}_{r}\sin t+X_{r}\cos t,\,\widetilde{X}_{r}\cos t+X_{r}\sin t)$ has the same distribution as $(X_{r},\widetilde{X}_{r})$ for every $t\in[0,\frac{\pi}{2}]$. Thus,
\begin{equation}\label{samplecovthmeq2}
\mathbb{E}\left\|\sum_{r=1}^{M}(X_{r}X_{r}^{*}-\widetilde{X}_{r}^{*}\widetilde{X}_{r})\right\|\leq\frac{\pi}{2}\mathbb{E}\left\|\sum_{r=1}^{M}X_{r}^{*}\widetilde{X}_{r}\right\|+\frac{\pi}{2}\mathbb{E}\left\|\sum_{r=1}^{M}\widetilde{X}_{r}^{*}X_{r}\right\|=\pi\mathbb{E}\left\|\sum_{r=1}^{M}\widetilde{X}_{r}^{*}X_{r}\right\|.
\end{equation}
Fix deterministic $D_{1},\ldots,D_{M}\in M_{d}(\mathbb{R})$. Then $\sum_{r=1}^{M}D_{r}X_{r}$ has jointly Gaussian entries and mean $0$. So by Theorem \ref{main1},
\begin{align*}
&\mathbb{E}\left\|\sum_{r=1}^{M}D_{r}X_{r}\right\|\\\lesssim_{\epsilon}&
\left\|\mathbb{E}\left(\sum_{r=1}^{M}D_{r}X_{r}\right)^{*}\left(\sum_{r=1}^{M}D_{r}X_{r}\right)\right\|^{\frac{1}{2}}+\left\|\mathbb{E}\left(\sum_{r=1}^{M}D_{r}X_{r}\right)\left(\sum_{r=1}^{M}D_{r}X_{r}\right)^{*}\right\|^{\frac{1}{2}}\\&+
d^{\epsilon}\left\|\mathbb{E}\left(\sum_{r=1}^{M}D_{r}X_{r}\right)\otimes\left(\sum_{r=1}^{M}D_{r}X_{r}\right)\right\|^{\frac{1}{2}}\\=&
\left\|\sum_{r=1}^{M}\mathbb{E}(X_{r}^{*}D_{r}^{*}D_{r}X_{r})\right\|^{\frac{1}{2}}+\left\|\sum_{r=1}^{M}\mathbb{E}(D_{r}^{*}X_{r}^{*}X_{r}D_{r})\right\|^{\frac{1}{2}}+d^{\epsilon}\left\|\sum_{r=1}^{M}\mathbb{E}[(D_{r}X_{r})\otimes(D_{r}X_{r})]\right\|^{\frac{1}{2}}\\\leq&
2\left(\sum_{r=1}^{M}\|D_{r}\|^{2}\mathbb{E}\|X_{r}\|^{2}\right)^{\frac{1}{2}}+d^{\epsilon}\left(\sum_{r=1}^{M}\|\mathbb{E}[(D_{r}X_{r})\otimes(D_{r}X_{r})]\|\right)^{\frac{1}{2}}
\end{align*}
But for every deterministic $D\in M_{d}(\mathbb{R})$ and every random matrix $X\in M_{d}(\mathbb{R})$ with $\mathbb{E}\|X\|^{2}<\infty$,
\begin{eqnarray*}
\|\mathbb{E}[(DX)\otimes(DX)]\|&=&\sup_{\substack{B\in M_{d}(\mathbb{R})\\\|B\|_{F}\leq 1}}\mathbb{E}[\mathrm{Tr}(DXB^{*})]^{2}\\&=&
\sup_{\substack{B\in M_{d}(\mathbb{R})\\\|B\|_{F}\leq 1}}\mathbb{E}[\mathrm{Tr}(XB^{*}D)]^{2}\\&\leq&
\|D\|^{2}\sup_{\substack{B\in M_{d}(\mathbb{R})\\\|B\|_{F}\leq 1}}\mathbb{E}[\mathrm{Tr}(XB^{*})]^{2}=\|D\|^{2}\|\mathbb{E}(X\otimes X)\|.
\end{eqnarray*}
Therefore,
\[\mathbb{E}\left\|\sum_{r=1}^{M}X_{r}D_{r}\right\|\lesssim_{\epsilon}2\left(\sum_{r=1}^{M}\|D_{r}\|^{2}\mathbb{E}\|X_{r}\|^{2}\right)^{\frac{1}{2}}+d^{\epsilon}\left(\sum_{r=1}^{M}\|D_{r}\|^{2}\|\mathbb{E}(X_{r}\otimes X_{r})\|\right)^{\frac{1}{2}}.\]
for all deterministic $D_{1},\ldots,D_{M}\in M_{d}(\mathbb{R})$. So
\begin{eqnarray*}
\mathbb{E}\left\|\sum_{r=1}^{M}X_{r}\widetilde{X}_{r}\right\|&\lesssim_{\epsilon}&
\left(\sum_{r=1}^{M}\mathbb{E}\|\widetilde{X}_{r}\|^{2}\mathbb{E}\|X_{r}\|^{2}\right)^{\frac{1}{2}}+d^{\epsilon}\left(\sum_{r=1}^{M}\mathbb{E}\|\widetilde{X}_{r}\|^{2}\|\mathbb{E}(X_{r}\otimes X_{r})\|\right)^{\frac{1}{2}}\\&=&
\left(\sum_{r=1}^{M}(\mathbb{E}\|X_{r}\|^{2})^{2}\right)^{\frac{1}{2}}+\left(\sum_{r=1}^{M}\mathbb{E}\|X_{r}\|^{2}d^{2\epsilon}\|\mathbb{E}(X_{r}\otimes X_{r})\|\right)^{\frac{1}{2}}\\&\leq&
\left(\sum_{r=1}^{M}(\mathbb{E}\|X_{r}\|^{2})^{2}\right)^{\frac{1}{2}}+\left(\sum_{r=1}^{M}[(\mathbb{E}\|X_{r}\|^{2})^{2}+d^{4\epsilon}\|\mathbb{E}(X_{r}\otimes X_{r})\|^{2}]\right)^{\frac{1}{2}}.
\end{eqnarray*}
By a Gaussian version of Kahane's inequality \cite{Kahane} or by concentration of $\|X\|$, we have $\mathbb{E}\|X\|^{2}\sim(\mathbb{E}\|X\|)^{2}$. Therefore,
\[\mathbb{E}\left\|\sum_{r=1}^{M}X_{r}\widetilde{X}_{r}\right\|\lesssim_{\epsilon}\left(\sum_{r=1}^{M}[(\mathbb{E}\|X_{r}\|)^{4}+d^{4\epsilon}\|\mathbb{E}(X_{r}\otimes X_{r})\|^{2}]\right)^{\frac{1}{2}}.\]
So by (\ref{samplecovthmeq1}) and (\ref{samplecovthmeq2}),
\[\mathbb{E}\left\|\sum_{r=1}^{M}(X_{r}^{*}X_{r}-\mathbb{E}(X_{r}^{*}X_{r}))\right\|\lesssim_{\epsilon}\left(\sum_{r=1}^{M}[(\mathbb{E}\|X_{r}\|)^{4}+d^{4\epsilon}\|\mathbb{E}(X_{r}\otimes X_{r})\|^{2}]\right)^{\frac{1}{2}}.\]
Thus, by Theorem \ref{main1}, the result follows.
\end{proof}
\begin{corollary}\label{samplecov2}
Suppose that $\mu$ is a probability measure on $\{B\in M_{d_{2}}(\mathbb{R})|\,B\text{ is positive semidefinite}\}$ and $\mathrm{Tr}(B)\geq\max(d_{1}^{\epsilon},d_{2}^{\epsilon})\|B\|$ $\mu$-a.s. Let $M\in\mathbb{N}$. Let $z_{1},\ldots,z_{Md_{1}}$ be i.i.d.~random vectors in $\mathbb{R}^{d_{2}}$ chosen according to $\int\mathcal{N}(0,B)\,d\mu(B)$, i.e., $\mathbb{P}(z_{1}\in\mathcal{S})=\int\mathbb{P}(B^{\frac{1}{2}}g\in\mathcal{S})\,d\mu(B)$ for all measurable $\mathcal{S}\subset\mathbb{R}^{d_{2}}$, where $g$ is a standard Gaussian on $\mathbb{R}^{d_{2}}$. Then
\begin{align*}
&\mathbb{E}\left\|\frac{1}{Md_{1}}\sum_{i=1}^{Md_{1}}z_{i}z_{i}^{T}-\int B\,d\mu(B)\right\|\\\lesssim_{\epsilon}&
\frac{1}{d_{1}\sqrt{M}}\left(d_{1}\left\|\int B\,d\mu(B)\right\|+\left(\mathbb{E}\max_{i\in[d_{1}]}[\mathrm{Tr}(B_{i})]^{2}\right)^{\frac{1}{2}}+\sqrt{d_{1}\log d_{2}}\left\|\int B^{2}\,d\mu(B)\right\|^{\frac{1}{2}}\right),
\end{align*}
where $B_{1},\ldots,B_{d_{1}}$ in $M_{d_{2}}(\mathbb{R})$ are i.i.d.~chosen according to $\mu$.
\end{corollary}
\begin{proof}
Fix $B_{1},\ldots,B_{Md_{1}}\in M_{d_{2}}(\mathbb{R})$. For each $r\in[M]$, let $X_{r}$ be a $d_{1}\times d_{2}$ random matrix with independent rows such that for $i\in[d_{1}]$, the $i$th row of $X_{r}$ is a centered Gaussian random vector with covariance matrix $B_{i+(r-1)d_{1}}\in M_{d_{2}}(\mathbb{R})$. Then using computations from the proof of Corollary \ref{indeprow}, we have, by Theorem \ref{samplecovthm},
\begin{align*}
&\mathbb{E}\left\|\sum_{r=1}^{M}(X_{r}^{*}X_{r}-\mathbb{E}(X_{r}^{*}X_{r}))\right\|\\\lesssim_{\epsilon}&
\left[\sum_{r=1}^{M}\left(\left\|\sum_{i=1}^{d_{1}}B_{i+(r-1)d_{1}}\right\|^{2}+\max_{i\in[d_{1}]}[\mathrm{Tr}(B_{i+(r-1)d_{1}})]^{2}+\max(d_{1}^{\epsilon},d_{2}^{\epsilon})\max_{i\in[d_{1}]}\|B_{i+(r-1)d_{1}}\|^{2}\right)\right]^{\frac{1}{2}}.
\end{align*}
By assumption, $z_{1},\ldots,z_{Md_{1}}$ are chosen as follows: first, choose i.i.d.~$B_{1},\ldots,B_{Md_{1}}$ in $M_{d_{2}}(\mathbb{R})$ according to $\mu$ and then for each $i\in[Md_{1}]$, take $z_{i}\sim\mathcal{N}(0,B_{i})$ Thus, conditioning on $B_{1},\ldots,B_{Md_{1}}$, we have
\begin{align*}
&\mathbb{E}\left(\left.\left\|\sum_{i=1}^{Md_{1}}(z_{i}z_{i}^{T}-B_{i})\right\|\;\right|\,B_{1},\ldots,B_{Md_{1}}\right)\\\lesssim_{\epsilon}&
\left[\sum_{r=1}^{M}\left(\left\|\sum_{i=1}^{d_{1}}B_{i+(r-1)d_{1}}\right\|^{2}+\max_{i\in[d_{1}]}[\mathrm{Tr}(B_{i+(r-1)d_{1}})]^{2}+\max(d_{1}^{\epsilon},d_{2}^{\epsilon})\max_{i\in[d_{1}]}\|B_{i+(r-1)d_{1}}\|^{2}\right)\right]^{\frac{1}{2}}\\\lesssim&
\left[\sum_{r=1}^{M}\left(\left\|\sum_{i=1}^{d_{1}}B_{i+(r-1)d_{1}}\right\|^{2}+\max_{i\in[d_{1}]}[\mathrm{Tr}(B_{i+(r-1)d_{1}})]^{2}\right)\right]^{\frac{1}{2}}\quad\mu\hyphen a.s.,
\end{align*}
since by assumption, $\mathrm{Tr}(B)\geq\max(d_{1}^{\epsilon},d_{2}^{\epsilon})\|B\|$ $\mu$-a.s.. So
\[\mathbb{E}\left\|\sum_{i=1}^{Md_{1}}(z_{i}z_{i}^{T}-B_{i})\right\|\lesssim_{\epsilon}\sqrt{M}\left(\mathbb{E}\left\|\sum_{i=1}^{d_{1}}B_{i}\right\|^{2}+\mathbb{E}\max_{i\in[d_{1}]}[\mathrm{Tr}(B_{i})]^{2}\right)^{\frac{1}{2}}.\]
By a modified version of \cite[Theorem 5.1(1)]{Troppelemappr},
\[\mathbb{E}\left\|\sum_{i=1}^{d_{1}}B_{i}\right\|^{2}\lesssim\left\|\sum_{i=1}^{d_{1}}\mathbb{E}B_{i}\right\|^{2}+(\log d_{2})^{2}\mathbb{E}\max_{i\in[d_{1}]}\|B_{i}\|^{2}=d_{1}^{2}\left\|\int B\,d\mu(B)\right\|^{2}+(\log d_{2})^{2}\mathbb{E}\max_{i\in[d_{1}]}\|B_{i}\|^{2}.\]
But $\mathrm{Tr}(B)\geq d_{2}^{\epsilon}\|B\|$ $\mu$-a.s.. Therefore,
\[\mathbb{E}\left\|\sum_{i=1}^{Md_{1}}(z_{i}z_{i}^{T}-B_{i})\right\|\lesssim_{\epsilon}\sqrt{M}\left(d_{1}\left\|\int B\,d\mu(B)\right\|+\left(\mathbb{E}\max_{i\in[d_{1}]}[\mathrm{Tr}(B_{i})]^{2}\right)^{\frac{1}{2}}\right).\]
But by \cite[Theorem 5.1(2)]{Troppelemappr},
\begin{eqnarray*}
\mathbb{E}\left\|\sum_{i=1}^{Md_{1}}(B_{i}-\mathbb{E}B_{i})\right\|&\lesssim&\sqrt{Md_{1}\log d_{2}}\left\|\int B^{2}\,d\mu(B)\right\|^{\frac{1}{2}}+(\log d_{2})\left(\mathbb{E}\max_{i\in[Md_{1}]}\|B_{i}-\mathbb{E}B_{i}\|^{2}\right)^{\frac{1}{2}}\\&\lesssim&
\sqrt{Md_{1}\log d_{2}}\left\|\int B^{2}\,d\mu(B)\right\|^{\frac{1}{2}}+(\log d_{2})\left(M\mathbb{E}\max_{i\in[d_{1}]}\|B_{i}\|^{2}\right)^{\frac{1}{2}}\\&\lesssim_{\epsilon}&
\sqrt{Md_{1}\log d_{2}}\left\|\int B^{2}\,d\mu(B)\right\|^{\frac{1}{2}}+\left(M\mathbb{E}\max_{i\in[d_{1}]}[\mathrm{Tr}(B_{i})]^{2}\right)^{\frac{1}{2}},
\end{eqnarray*}
since $\mathrm{Tr}(B)\geq d_{2}^{\epsilon}\|B\|$ $\mu$-a.s.. Therefore,
\begin{align*}
&\mathbb{E}\left\|\sum_{i=1}^{Md_{1}}(z_{i}z_{i}^{T}-\mathbb{E}B_{i})\right\|\\\lesssim_{\epsilon}&
\sqrt{M}\left(d_{1}\left\|\int B\,d\mu(B)\right\|+\left(\mathbb{E}\max_{i\in[d_{1}]}[\mathrm{Tr}(B_{i})]^{2}\right)^{\frac{1}{2}}+\sqrt{d_{1}\log d_{2}}\left\|\int B^{2}\,d\mu(B)\right\|^{\frac{1}{2}}\right).
\end{align*}
Thus, the result follows.
\end{proof}
\begin{corollary}\label{samplecov3}
Suppose that $\mu$ is a probability measure on $\{B\in M_{d_{2}}(\mathbb{R})|\,B\text{ is positive semidefinite}\}$ and $\mathrm{Tr}(B)\leq t$ and $\displaystyle\|B\|\leq\frac{t}{\max(d_{1}^{\epsilon},d_{2}^{\epsilon})}$ $\mu$-a.s., where $t\in[d_{2}]$ is fixed. Let $M\in\mathbb{N}$. Let $z_{1},\ldots,z_{Md_{1}}$ be i.i.d.~random vectors in $\mathbb{R}^{d_{2}}$ chosen according to $\int\mathcal{N}(0,B)\,d\mu(B)$, i.e., $\mathbb{P}(z_{1}\in\mathcal{S})=\int\mathbb{P}(B^{\frac{1}{2}}g\in\mathcal{S})\,d\mu(B)$ for all measurable $\mathcal{S}\subset\mathbb{R}^{d_{2}}$, where $g$ is a standard Gaussian on $\mathbb{R}^{d_{2}}$. Then
\[\mathbb{E}\left\|\frac{1}{Md_{1}}\sum_{i=1}^{Md_{1}}z_{i}z_{i}^{T}-\int B\,d\mu(B)\right\|\\\lesssim_{\epsilon}\frac{1}{\sqrt{M}}\left(\left\|\int B\,d\mu(B)\right\|+\frac{t}{d_{1}}\right).\]
\end{corollary}
\begin{proof}
By Corollary \ref{samplecov2},
\[\mathbb{E}\left\|\frac{1}{Md_{1}}\sum_{i=1}^{Md_{1}}z_{i}z_{i}^{T}-\int B\,d\mu(B)\right\|\\\lesssim_{\epsilon}\frac{1}{\sqrt{M}}\left(\left\|\int B\,d\mu(B)\right\|+\frac{t}{d_{1}}+\sqrt{\frac{\log d_{2}}{d_{1}}}\left\|\int B^{2}\,d\mu(B)\right\|^{\frac{1}{2}}\right).\]
Since
\[\left\|\int B^{2}\,d\mu(B)\right\|\leq\left\|\int\|B\|B\,d\mu(B)\right\|\leq\frac{t}{\max(d_{1}^{\epsilon},d_{2}^{\epsilon})}\left\|\int B\,d\mu(B)\right\|,\]
we have
\[\sqrt{\frac{\log d_{2}}{d_{1}}}\left\|\int B^{2}\,d\mu(B)\right\|^{\frac{1}{2}}\lesssim_{\epsilon}\sqrt{\frac{t}{d_{1}}}\left\|\int B\,d\mu(B)\right\|^{\frac{1}{2}}\leq\frac{t}{d_{1}}+\left\|\int B\,d\mu(B)\right\|.\]
Thus, the result follows.
\end{proof}
\begin{remark}
Suppose that $\mu$ is a probability measure on $\{B\in M_{d_{2}}(\mathbb{R})|\,B\text{ is positive semidefinite}\}$ and $\mathrm{Tr}(B)=t$ and $\displaystyle\|B\|\leq\frac{t}{d_{2}^{\epsilon}}$ $\mu$-a.s., where $t\in[d_{2}]$ is fixed. Let $0<\gamma<1$.\\
How many samples $z_{1},\ldots,z_{m}$ do we need so that
\[\mathbb{E}\left\|\frac{1}{m}\sum_{i=1}^{m}z_{i}z_{i}^{T}-\int B\,d\mu(B)\right\|\leq\gamma\left\|\int B\,d\mu(B)\right\|\,?\]
Since $\displaystyle\mathbb{E}\left\|\frac{1}{m}\sum_{i=1}^{m}z_{i}z_{i}^{T}\right\|\geq\frac{1}{m}\mathbb{E}\|z_{1}\|_{2}^{2}=\frac{1}{m}\int\mathrm{Tr}(B)\,d\mu(B)=\frac{t}{m}$, we need at least $\displaystyle\frac{t}{2\|\int B\,d\mu(B)\|}$ samples. Let $d_{1}=\left\lceil\frac{t}{2\|\int B\,d\mu(B)\|}\right\rceil$. Note that $d_{1}\in[d_{2}]$. By Corollary \ref{samplecov3}, we have
\[\mathbb{E}\left\|\frac{1}{Md_{1}}\sum_{i=1}^{Md_{1}}z_{i}z_{i}^{T}-\int B\,d\mu(B)\right\|\\\lesssim_{\epsilon}\frac{1}{\sqrt{M}}\left(\left\|\int B\,d\mu(B)\right\|+\frac{t}{d_{1}}\right)\leq\frac{3}{\sqrt{M}}\left\|\int B\,d\mu(B)\right\|.\]
Thus, the answer to the above question is between $d_{1}$ and $\displaystyle\frac{C_{\epsilon}}{\gamma^{2}}d_{1}$ samples, where $C_{\epsilon}>0$ is a constant that depends only on $\epsilon$.
\end{remark}
{\bf Acknowledgement:} The authors are grateful to Ramon van Handel and Joel Tropp for many insightful comments and discussions. The second author is supported by NSF DMS-1856221.

\end{document}